\documentclass[12pt]{article}
 
 \usepackage{amsmath,amssymb,amscd,amsthm,esint}
 
\usepackage{graphics,amsmath,amssymb,amsthm,mathrsfs}

\newtheorem{theorem}{Theorem}[section]
\newtheorem{prop}[theorem]{Proposition}
\newtheorem{lemma}[theorem]{Lemma}

\newtheorem{remark}[theorem]{Remark}

\def\xxint#1#2#3{{\setbox0=\hbox{$#1{#2#3}{\int}$}
  \vcenter{\hbox{$#2#3$}}\kern-.5\wd0}}

\def \rd {{\mathbb R}^d}

\def\loc{\text{\rm loc}}

\def\varep{\varepsilon}

\newcommand{\average}{-\!\!\!\!\!\!\int}

\begin{document}

\title
{\bf Convergence Rates and H\"older Estimates \\ in Almost-Periodic Homogenization \\ of
Elliptic Systems}

\author{Zhongwei Shen\thanks{Supported in part by NSF grant DMS-1161154.}}


\date{ }

\maketitle

\begin{abstract}

For a family of second-order elliptic systems in divergence form with 
rapidly oscillating almost-periodic coefficients, we obtain estimates for approximate correctors
in terms of a function that quantifies the almost periodicity of the coefficients.
The results are used to investigate the problem of convergence rates.
We also establish uniform H\"older estimates for the Dirichlet problem
in a bounded $C^{1, \alpha}$ domain.

\bigskip



 \end{abstract}

 \medskip


\section {Introduction and statement of main results}
\setcounter{equation}{0}

In this paper we consider a family of second-order elliptic operators in divergence form
with rapidly oscillating {\it almost-periodic} coefficients,
\begin{equation}\label{operator}
\mathcal{L}_\varep =-\text{div} \left(A(x/\varep)\nabla \right)
=-\frac{\partial}{\partial x_i}
\left[ a_{ij}^{\alpha\beta} \left(\frac{x}{\varep}\right)\frac{\partial}{\partial x_j}\right], 
\quad \varep>0.
\end{equation}
We will assume that $ A(y)=\big( a_{ij}^{\alpha\beta} (y)\big)$
with $1\le i,j\le d$ and $1\le\alpha,\beta\le m$ is real and satisfies the ellipticity condition
\begin{equation}\label{ellipticity}
\mu |\xi|^2 \le a_{ij}^{\alpha\beta} (y)
\xi_i^\alpha\xi_j^\beta \le \frac{1}{\mu} |\xi|^2 
\quad
\text{ for } y\in \rd \text{ and } \xi =(\xi_i^\alpha)\in \mathbb{R}^{d\times m},
\end{equation}
where $\mu>0$ (the summation convention is used throughout the paper).
We further assume that $A=A(y)$ is uniformly almost-periodic in $\rd$; 
i.e., $A$ is the uniform limit of a sequence of trigonometric polynomials in $\rd$.

Let $\Omega$ be a bounded Lipschitz domain in $\rd$.
Let $u_\varep \in H^1(\Omega; \mathbb{R}^m)$ be the weak solution of the Dirichlet problem:
\begin{equation}\label{DP}
\mathcal{L}_\varep (u_\varep )= F \quad \text{ in } \Omega \quad \text{ and } \quad
u_\varep =g \quad \text{ on } \partial\Omega,
\end{equation}
where $F\in H^{-1}(\Omega; \mathbb{R}^m)$ and $g\in H^{1/2}(\partial\Omega; \mathbb{R}^m)$.
Under the ellipticity condition (\ref{ellipticity}) and almost periodicity condition on $A$,
it is known that $u_\varep \to u_0$ weakly in $H^1(\Omega;\mathbb{R}^m)$ 
and thus strongly in $L^2(\Omega;\mathbb{R}^m)$, 
as $\varep \to 0$.
Furthermore, the function $u_0$ is the solution of
\begin{equation}\label{DP-H}
\mathcal{L}_0 (u_0 )= F \quad \text{ in } \Omega \quad \text{ and } \quad
u_0 =g \quad \text{ on } \partial\Omega,
\end{equation}
where $\mathcal{L}_0 =-\text{\rm div} \big(\widehat{A} \nabla\big)$
is a second-order elliptic operator with constant coefficients, uniquely determined by $A(y)$.
As in the periodic case (see e.g. \cite{bensoussan-1978}), 
the constant matrix $\widehat{A}=\big( \widehat{a}_{ij}^{\alpha\beta}\big)$
is called the homogenized matrix for $A$
and $\mathcal{L}_0$ the homogenized operator for $\mathcal{L}_\varep$.
In this paper we shall be interested in quantitative homogenization results
as well as uniform estimates for solutions of (\ref{DP}).

Homogenization of elliptic equations with rapidly oscillating almost-periodic or random coefficients 
was studied first by S. M. Kozlov in \cite{Kozlov-1979, Kozlov-1980} and by
G.C. Papanicolaou and S.R.S. Varadhan in \cite{Papanicolaou-1979}.
In particular, the $o(1)$ convergence rate of $u_\varep -u_0$ in $C^\sigma(\overline{\Omega})$
for some $\sigma>0$ was obtained in \cite{Kozlov-1979} for a scalar second-order elliptic equation
 in divergence form with almost-periodic coefficients.
 Under some additional condition on the frequencies in the spectrum of $A(y)$,
 the sharp $O(\varep)$ rate in $C(\overline{\Omega})$ 
 was proved in \cite{Kozlov-1979} for operators with sufficiently smooth quasi-periodic coefficients.
 It is known that without additional structure conditions on $A(y)$, the $O(\varep)$ rate cannot be expected in general (see \cite{Bondarenko-2005} for  some interesting results in the 1-d case).

In contrast to the periodic case,
 the equation for the exact correctors $\chi (y)$,
\begin{equation}\label{c}
-\text{\rm div} \big(A (y)\nabla \chi (y) \big) =\text{\rm div} \big(A(y) \nabla P(y)\big)\quad \text{ in } \rd,
\end{equation}
may not be solvable in the almost-periodic (or random) setting for linear functions $P(y)$.
In \cite{Kozlov-1979}   solutions $\chi (y)$ of (\ref{c}) with sub-linear growth and almost-periodic
gradient were constructed, and as a result, homogenization was obtained,
 for operators with trigonometric polynomial coefficients by a lifting method. The homogenization
 result for the general case follows by an approximation argument.
  A different approach, which also gives the homogenization of the second-order elliptic equations
 with random coefficients, is to formulate and solve an abstract auxiliary equation 
 in a Hilbert space for $\psi (y)=\nabla \chi (y)$.
 We outline this approach in Section 2 and refer the reader to \cite{Jikov-1994}
 for a detailed presentation and references.
 
Another approach to homogenization involves
 the use of the so-called approximate correctors \cite{Papanicolaou-1979, Kozlov-1980}.
 Under certain mixing conditions,
the approach has been employed successfully to establish quantitative homogenization results
for second-order linear elliptic equations and systems in divergence form with random coefficients  in
 \cite{Yurinskii-1986, Yurinskii-1990, Bourgeat-2004}.
 For nonlinear second-order elliptic equations
 and Hamilton-Jacobi equations, we refer the reader to
 \cite{Caffarelli-2010, Armstrong-2014, Smart-2014} for recent advances and references on quantitative homogenization results.
 We point out that the almost-periodic case, which does not satisfy the mixing conditions generally imposed
 in the random case, is studied 
 in  \cite{Caffarelli-2010, Armstrong-2014}.
 We also mention that sharp quantitative results 
 were obtained recently in \cite{Otto-2011, Otto-2012, Gloria-2014} 
 for stochastic homogenization of discrete linear elliptic equations in divergence form. 
 
In this paper we carry out  a quantitative study of the approximate correctors $\chi_T
=\big(\chi_{T, j}^\beta\big)$ for $\mathcal{L}_\varep$
 in (\ref{operator}), where, for $1\le j\le d$ and $1\le \beta\le m$, $u=\chi_{T, j}^\beta$ is defined by
 \begin{equation}\label{ap-c-0}
 -\text{\rm div} \big( A(y)\nabla u\big)
 +T^{-2}  u =\text{\rm div} \big( A(y)\nabla P_j^\beta (y)\big) \quad \text{ in } \rd,
 \end{equation}
 and $P_j^\beta (y) =y_j (0, \cdots, 1, \cdots, 0)$ with $1$ in the $\beta^{th}$ position.
 Among other things, we will prove that for $T\ge 1$ and $\sigma \in (0,1)$,
 \begin{equation}\label{ap-estimate-0}
 T^{-1} \| \chi_T \|_{L^\infty (\rd)} \le C_\sigma\, \Theta_\sigma (T),
 \end{equation}
 \begin{equation}\label{ap-estimate-0-0}
 |\chi_T (x) -\chi_T (y)|\le C_\sigma \, T^{1-\sigma} |x-y|^\sigma \quad \text{ for any } x, y\in \rd,
 \end{equation}
 and for $0<r\le T$,
 \begin{equation}\label{ap-estimate-1}
 \sup_{x\in \rd}
 \left(\average_{B(x,r)} |\nabla \chi_T|^2\right)^{1/2}
 \le C_\sigma \left(\frac{T}{r} \right)^\sigma,
 \end{equation}
 where $C_\sigma$ depends only on $d$, $m$, $\sigma$ and $A$. 
 The  continuous function $\Theta_\sigma (T)$, which is decreasing and
 converges to zero as $T\to \infty$, is defined by
\begin{equation}\label{Theta}
\Theta_\sigma (T)=\inf_{0<R\le T} \left\{ \rho (R) +\left(\frac{R}{T}\right)^\sigma \right\},
\end{equation}
where
\begin{equation}\label{rho}
\rho (R)=\sup_{y\in \rd} \inf_{\substack{z\in \rd\\ |z|\le R} }
\| A(\cdot +y) -A(\cdot+z)\|_{L^\infty(\rd)}
\end{equation}
is a decreasing and continuous function that quantifies the almost periodicity of $A$.
Indeed, a bounded continuous function $A$ in $\rd$ is uniformly almost-periodic
if and only if $\rho (R)\to 0$ as $R\to \infty$.

With the estimates (\ref{ap-estimate-0}), (\ref{ap-estimate-0-0}) and (\ref{ap-estimate-1}) at our disposal,
we obtain the following theorems on the convergence rates.
Our results in Theorems \ref{main-theorem-2} and \ref{main-theorem-3} are new even in the scalar case $m=1$.

\begin{theorem}\label{main-theorem-1}
Suppose that $A(y)=\big(a_{ij}^{\alpha\beta} (y)\big)$satisfies the ellipticity condition (\ref{ellipticity})
and is uniformly almost-periodic in $\rd$.
Let $p>d$, $\sigma \in (0,1)$, and $\Omega$ be a bounded $C^{1, \alpha}$ domain in $\rd$ for some $\alpha>0$.
Then there exists a modulus  $\eta: (0,1]\to [0, \infty)$, which depends only on $A$ and $\sigma$,
such that $\lim_{ t\to 0} \eta (t)=0$ and
\begin{equation}\label{1.1-0}
\| u_\varep -u_0\|_{C^\sigma (\overline{\Omega})} \le C\, \eta (\varep) \, \| u_0\|_{W^{2, p}(\Omega)}
\end{equation}
for $\varep \in (0,1)$,
whenever $u_\varep\in H^1(\Omega)$ is the weak solution of (\ref{DP})
and $u_0\in W^{2, p} (\Omega)$ the solution of (\ref{DP-H}).
Furthermore, we have
\begin{equation}\label{1.1-1}
\| u_\varep -u_0 -\varep \chi_T (x/\varep) \nabla u_0\|_{H^1(\Omega)}
\le C\, \eta (\varep) \, \| u_0 \|_{W^{2, p} (\Omega)},
\end{equation}
where $T=\varep^{-1}$ and $\chi_T(y)$ denotes the approximate corrector defined by (\ref{ap-c-0}).
The constants $C$ in (\ref{1.1-0}) and (\ref{1.1-1}) depend only on $\Omega$, $p$, $\sigma$ and $A$.
\end{theorem}

The next theorem gives more precise rates of convergence, provided 
$\rho(R)$ decays fast enough so that $\int_1^\infty  \frac{\rho(r)}{r}\, dr<\infty$.

\begin{theorem}\label{main-theorem-2}
Under the same assumptions as in Theorem \ref{main-theorem-2}, the following estimates hold:
\begin{equation}\label{1.2-0}
\| u_\varep -u_0\|_{L^2 (\Omega)}
\le C\, \|u_0\|_{W^{2, p}(\Omega)} \left\{
\int^\infty_{\frac{1}{2\varep}} \frac{\Theta_\sigma (r) }{r}\, dr
+ \big[ \Theta_1 (\varep^{-1})\big]^\sigma\right\},
\end{equation}
and
\begin{equation}\label{1.2-1}
\| u_\varep -u_0 -\varep \chi_T (x/\varep) \nabla u_0\|_{H^1(\Omega)}
\le C\, \| u_0 \|_{W^{2, p} (\Omega)} \left\{ \int^\infty_{\frac{1}{2\varep}} \frac{\Theta_\sigma (r)}{r}\, dr
+\big[ \Theta_1 (\varep^{-1})\big]^{\frac{\sigma}{2}}\right\}
\end{equation}
for any $\sigma \in (0,1)$, where $T=\varep^{-1}$ and
$C$ depends only on $\Omega$, $A$, $p$ and $\sigma$.
\end{theorem}

\begin{remark}\label{remark-1.1}
{\rm
By taking $R=\sqrt{T}$ in (\ref{Theta}), we obtain $\Theta_\sigma (T)
\le \rho (\sqrt{T} )+T^{-\frac{\sigma}{2}}$ for $T\ge 1$.
It follows that
\begin{equation}\label{additional}
\text{ if }
\int_1^\infty \frac{\rho (r)}{r}\, dr<\infty, \text{ then } 
\int_1^\infty \frac{\Theta_\sigma (r)}{r}\, dr<\infty
\end{equation}
for any $\sigma \in (0,1]$.
It is not clear whether estimates (\ref{1.2-0}) and (\ref{1.2-1}) are sharp.
However, let us suppose that there exist $\tau>0$ and $C>0$ such that
\begin{equation}\label{1.3-0}
\rho (R) \le C \, R^{-\tau} \quad \text{ for all } R\ge 1.
\end{equation}
Then, for $T\ge 1$,
$$
\Theta_\sigma  (T) \le C\, T^{-\frac{\sigma \tau}{\sigma +\tau}}.
$$
It follows from (\ref{1.2-0}) that
$$
\| u_\varep -u_0\|_{L^2 (\Omega)} \le C \, \varep^{\frac{\sigma\tau}{\sigma +\tau}} \, \| u_0\|_{W^{2, p} (\Omega)}.
$$
Since $\sigma \in (0,1)$ is arbitrary, this gives
\begin{equation}\label{1.3-3}
\| u_\varep -u_0\|_{L^2 (\Omega)} \le C_\gamma \, \varep^\gamma \, \| u_0\|_{W^{2, p} (\Omega)}
\quad \text{ for any } 0<\gamma<\frac{\tau}{\tau +1}.
\end{equation}
Similarly, one may deduce from (\ref{1.2-1}) that
\begin{equation}\label{1.3-4}
\| u_\varep -u_0 -\varep \chi_T (x/\varep) \nabla u_0\|_{H^1(\Omega)}
\le C_\gamma \, \varep^\gamma \,\| u_0\|_{W^{2, p} (\Omega)}
\end{equation}
for any $0<\gamma <\frac{\tau}{2 (\tau +1)}$.
It is interesting to point out that if $A$ is periodic, then $\rho(R)=0$ for $R$ large and thus
the condition (\ref{1.3-0}) holds for any $\tau>1$.
Consequently, estimates (\ref{1.3-3}) and (\ref{1.3-4}) yield convergence rates $O(\varep^{1-\delta})$ and
$O(\varep^{\frac12-\delta})$ for any $\delta>0$ in $L^2(\Omega)$ and $H^1(\Omega)$ respectively, 
which are near optimal.
Also note that under the condition (\ref{1.3-0}), our estimate (\ref{ap-estimate-0}) gives
\begin{equation}\label{1.3-5}
 \| \chi_T\|_{L^\infty} \le C_\delta\, T^{\frac{1}{\tau +1} +\delta}
\end{equation}
for any $\delta>0$, while one has $\|\chi_T\|_{L^\infty}\le C$, if $A$ is periodic.
Section 8 contains some examples of quasi-periodic functions for which the condition (\ref{1.3-0})
is satisfied.
}
\end{remark}

In this paper we also establish the uniform H\"older estimates for the Dirichlet problem (\ref{DP}).

\begin{theorem}\label{main-theorem-3}
Suppose that $A(y)=\big(a_{ij}^{\alpha\beta} (y)\big)$satisfies the ellipticity condition (\ref{ellipticity})
and is uniformly almost-periodic in $\rd$.
Let  $\Omega$ be a bounded $C^{1, \alpha}$ domain in $\rd$ for some $\alpha>0$.
Let $u_\varep $ be a weak solution of
\begin{equation}\label{1.10-0}
\mathcal{L}_\varep (u_\varep)=F +\text{\rm div} (f) \quad \text{ in } \Omega \quad 
\text{ and } \quad u_\varep =g \quad \text{ on } \partial\Omega.
\end{equation}
Then, for any $\sigma\in (0,1)$,
\begin{equation}\label{1.10-1}
\aligned
\| u_\varep\|_{C^\sigma (\overline{\Omega})}
\le C \bigg\{ \| g\|_{C^\sigma (\partial\Omega)}
&+\sup_{\substack{ x\in \Omega \\ 0<r<r_0}} r^{2-\sigma} 
\average_{B(x,r)\cap\Omega} |F|\\
& +\sup_{\substack{x\in \Omega\\ 0<r<r_0}}
r^{1-\sigma} \left(\average_{B(x,r)\cap\Omega}
|f|^2\right)^{1/2} \bigg\},
\endaligned
\end{equation}
where $r_0=\text{\rm diam} (\Omega)$ and $C$ depends only on $\sigma$, $A$ and $\Omega$.
\end{theorem}

We now describe the outline of this paper as well as
some of key ideas used in the proof of its main results.
In Section 2 we give a brief review of the homogenization of second-order elliptic systems
with almost-periodic coefficients, based on an auxiliary equation in $B^2(\rd)$, the
Bezikovich space of almost-periodic functions.
We also prove a homogenization theorem  (Theorem \ref{compactness-theorem}) for a sequence of operators 
$\{ -\text{\rm div} \big( B_\ell (x/\varep_\ell)\nabla \big)\}$,
where $\varep_\ell \to 0$ and $\{ B_\ell (y)\}$ are obtained from $A(y)$ through rotations and
translations. With this theorem a compactness argument
 is used in Sections 3 and 4 to establish the uniform interior and boundary
H\"older estimates for local solutions of $\mathcal{L}_\varep (u_\varep)=F +\text{\rm div} (f)$.
The proof of Theorem \ref{main-theorem-3} is given in Section4.
 We mention that the compactness argument, 
 which originated from the regularity theory in the calculus of variations and minimal surfaces,
was introduced to the study of homogenization problems  
by M. Avellaneda and F. Lin \cite{AL-1987, AL-1989-II}.
It was used recently  in \cite{KLS1} to establish the Lipschitz estimates for the  Neumann problem
in periodic homogenization. Also see related work in \cite{shen-2008, song-2012, shen-2014}.
In the almost-periodic setting the compactness argument was used in \cite{Dungey-2001}
to obtain the interior H\"older estimate  for operators with complex coefficients.
However, we point out that some version of Theorem \ref{compactness-theorem}
seems to be necessary to ensure that the constants are independent of the centers of balls.

The approximate correctors $\chi_T$ are constructed in Section 5, while
estimates (\ref{ap-estimate-0}), (\ref{ap-estimate-0-0}) and (\ref{ap-estimate-1}) are established in Section 6.
The proof of (\ref{ap-estimate-0-0}) and (\ref{ap-estimate-1})
relies on  the uniform H\"older estimates for $\mathcal{L}_\varep$.
We will also show that
\begin{equation}\label{90}
| \chi_T (x) -\chi_T ( y) |\le C\, T \, \| A(\cdot+ x) -A(\cdot +y)\|_{L^\infty } \quad \text{ for any } x, y\in \rd.
\end{equation}
The estimate (\ref{ap-estimate-0}) follows from (\ref{90}) and (\ref{ap-estimate-0-0})
in a manner somewhat similar to the case of Hamilton-Jacobi equations
in the almost-periodic setting \cite{Ishii-2000, Lions-2005, Armstrong-2014}.

Theorems \ref{main-theorem-1} and \ref{main-theorem-2} are proved in Section 7.
Here we follow an approach for the periodic case by considering 
$$
w_\varep =u_\varep  (x) - u_0 (x) -\varep \chi_T (x/\varep) \nabla u_0 (x) +v_\varep (x),
$$
where $T=\varep^{-1}$ and $v_\varep$ is the weak solution of
$\mathcal{L}_\varep (v_\varep)=0$ in $\Omega$ and
$v_\varep =\varep \chi_T (x/\varep)\nabla u_0 (x)$ on $\partial\Omega$.
We are able to show that
\begin{equation}\label{100}
\| w_\varep\|_{H^1(\Omega)}
\le C_\sigma \, \Big\{ \Theta_\sigma (T) +\langle |\psi -\nabla \chi_T |\rangle \Big\} \| u_0\|_{W^{2,2} (\Omega)}
\end{equation}
for any $\sigma \in (0,1)$, where $\psi$ is the limit of $\nabla\chi_T$ in $B^2(\rd)$ as $T\to \infty$.
In the periodic case one of the key steps is to write $\widehat{A}- A(y) -A(y)\nabla \chi (y)$
as a divergence of some bounded periodic function.
In the almost-periodic setting this will be replaced by solving the equation 
\begin{equation}\label{101}
-\Delta u+ T^{-2} u =B_T -\langle B_T\rangle \quad \text{ in }\rd,
\end{equation}
where $B_T (y)=\widehat{A} -A (y) -A (y) \nabla \chi_T (y)$.
The same ideas for proving (\ref{ap-estimate-0})-(\ref{ap-estimate-1}) are used
to obtain the desired estimates for $\| u\|_{L^\infty}$ and $\|\nabla u\|_{L^\infty}$
in terms of the function $\Theta_\sigma (T)$.
Finally, in Section 8 we consider the case of quasi-periodic coefficients
and provide some sufficient conditions on the frequencies of $A(y)$
for the estimate (\ref{1.3-0}) on $\rho(R)$.

Throughout this paper, unless indicated otherwise, we always assume that 
$A=\big(a_{ij}^{\alpha\beta}\big)$ satisfies the ellipticity condition
(\ref{ellipticity}) and is uniformly almost-periodic in $\rd$.
We will use $\average_E f=\frac{1}{|E|} \int_E f$ to denote the $L^1$ average of $f$
over $E$, and  $C$ to denote constants that depend on $A(y)$, $\Omega$ and other
relevant parameters, but never on $\varep$ or $T$.


\section{Homogenization and compactness}
\setcounter{equation}{0}

This section contains a brief review of homogenization theory of elliptic systems 
with almost-periodic coefficients. We refer the reader to \cite[pp.238-242]{Jikov-1994} for a detailed presentation.
We also prove a homogenization theorem for a sequence of operators obtained from 
$\mathcal{L}_\varep$ through translations and rotations.

Let $\text{\rm Trig}(\rd)$ denote the set of (real) trigonometric polynomials in $\rd$.
A bounded continuous function $f$ in $\rd$ is said to be uniformly almost-periodic  
(or almost-periodic in the sense of Bohr),
if $f$ is a limit of  a sequence of functions in $\text{\rm Trig} (\rd)$ with respect to the norm
$\| f\|_{L^\infty}$. A function $f$ in $L^2_\loc(\rd)$ is said to belong to $B^2(\rd)$
 if $f$ is a limit of a sequence of
functions in $\text{\rm Trig}(\rd)$ with respect to the semi-norm
\begin{equation}\label{B-2-norm}
\| f\|_{B^2} =\limsup_{R\to\infty}
\left\{\average_{B(0,R)}  |f|^2\right\}^{1/2}.
\end{equation}
Functions in $B^2(\rd)$ are said to be almost-periodic in the sense of Bezikovich.
It is not hard to see that if $f\in B^2(\rd)$ and $g$ is uniformly almost-periodic,
then $fg\in B^2(\rd)$. 

Let $f\in L^1_\loc(\rd)$.
A number $\langle f \rangle $ is called the mean value of $f$ if
\begin{equation}\label{mean-value}
\lim_{\varep\to 0^+} \int_{\rd} f(x/\varep)\varphi (x)\, dx  = \langle f \rangle \int_{\rd} \varphi
\end{equation}
for any $\varphi\in C_0^\infty(\rd)$.
If $f\in L^2_{\loc}(\rd)$ and $\| f\|_{B^2}<\infty$, the existence of $\langle f\rangle $ is equivalent to the
condition that as $\varep\to 0$, $f(x/\varep)\rightharpoonup \langle f\rangle $ weakly in $L^2_{\loc} (\rd)$, i.e.
$f(x/\varep)\rightharpoonup \langle f \rangle $ weakly in $L^2(B(0,R))$ for any $R>1$.
In this case one has
$$
\langle f \rangle =\lim_{L\to \infty} \average_{B(0, L)} f.
$$
It is known that if $f, g\in B^2(\rd)$, then $fg$ has the mean value. Furthermore, under the equivalent
relation that $f\sim g$ if $\| f-g\|_{B^2}=0$,
 the set $B^2(\rd)/\sim$ is a Hilbert space with the inner product defined by
$(f,g)=\langle fg \rangle $.

A function $f=(f_i^\alpha)$ in $\text{\rm Trig} (\rd; \mathbb{R}^{d\times m})$ is called potential if
there exists $g=(g^\alpha)\in \text{\rm Trig} (\rd; \mathbb{R}^{m})$ such that
$f_i^\alpha ={\partial g^\alpha}/{\partial x_i}$.
A function $f=(f_i^\alpha)$ in $\text{\rm Trig} (\rd; \mathbb{R}^{d\times m})$ is called solenoidal
if $\partial f_i^\alpha/\partial x_i=0$ for $1\le \alpha\le m$.
Let $V^2_{\text{\rm pot}}$ (resp. $V^2_{\rm sol}$) denote the closure of potential  (resp. solenoidal)
trigonometric polynomials with mean value zero
in $B^2(\rd; \mathbb{R}^{d\times m})$.
Then
\begin{equation}\label{decomp}
B^2(\rd; \mathbb{R}^{d\times m}) =V^2_{\rm pot} \oplus V^2_{\rm sol}\oplus \mathbb{R}^{d\times m}.
\end{equation}
By the Lax-Milgram Theorem and the ellipticity condition (\ref{ellipticity}), for any $1\le j\le d$
and $1\le \beta\le m$, there exists a unique $\psi_j^\beta =(\psi_{ij}^{\alpha\beta})
\in V_{\rm pot}^2$ such that
\begin{equation}\label{c-equation}
\langle a_{ik}^{\alpha\gamma}\,  \psi_{kj}^{\gamma\beta} \, \phi_i^\alpha\rangle 
=-\langle a_{ij}^{\alpha\beta} \, \phi_i^\alpha\rangle 
\quad \text{ for any } \phi=(\phi_i^\alpha)\in V_{\rm pot}^2.
\end{equation}
Let
\begin{equation}\label{homo-coefficient}
\widehat{a}_{ij}^{\alpha\beta}
=\langle a_{ij}^{\alpha\beta}\rangle +\langle a_{ik}^{\alpha\gamma} \, \psi_{kj}^{\gamma\beta}\rangle .
\end{equation}
and $\widehat{A} =\big(\widehat{a}_{ij}^{\alpha\beta}\big)$.
Then
\begin{equation}\label{ellipticity-1}
\mu |\xi|^2 \le \hat{a}_{ij}^{\alpha\beta} \xi_i^\alpha \xi_j^\beta \le \mu_1 |\xi|^2
\end{equation}
for any $\xi =(\xi_i^\alpha)\in \mathbb{R}^{d\times m}$,
where $\mu_1$ depends only on $d$, $m$ and $\mu$.
It is also known  that $\widehat{A^*} =\big(\widehat{A}\big)^*$,
where $A^*$ denotes the adjoint of $A$, i.e., $A^*=\big(b_{ij}^{\alpha\beta}\big)$
with $b_{ij}^{\alpha\beta}=a_{ji}^{\beta\alpha}$.

As the following theorem shows,
the homogenized operator  for $\mathcal{L}_\varep$ is given by $\mathcal{L}_0
=-\text{\rm div} \big(\widehat{A} \nabla\big)$.

\begin{theorem}\label{homo-theorem}
Let $\Omega$ be a bounded Lipschitz domain in $\rd$ and
 $F\in H^{-1}(\Omega; \mathbb{R}^m)$.
 Let  $u_\varep\in H^1 (\Omega; \mathbb{R}^{m})$
be a weak solution of $\mathcal{L}_\varep (u_\varep) =F$ in $\Omega$.
Suppose that $u_\varep \rightharpoonup u_0$ weakly in $H^1(\Omega; \mathbb{R}^m)$.
Then $A(x/\varep)\nabla u_\varep \rightharpoonup \widehat{A}\nabla u_0$ weakly in $L^2(\Omega; \mathbb{R}^{dm})$.
Consequently, if $f\in H^{1/2}(\partial\Omega; \mathbb{R}^m)$ and
 $u_\varep$ is the unique weak solution in $H^1(\Omega;\mathbb{R}^m)$ of 
the Dirichlet problem: \, $\mathcal{L}_\varep (u_\varep) =F$ in $\Omega$ and
$u_\varep =f$ on $\partial\Omega$, 
then, as $\varep\to 0$, $u_\varep \to u_0$ weakly in $H^1(\Omega;\mathbb{R}^m)$ and
strongly in $L^2(\Omega; \mathbb{R}^m)$, where $u_0$ is the unique weak solution
in $H^1(\Omega;\mathbb{R}^m)$ of  the Dirichlet problem:\ 
$\mathcal{L}_0 (u_0)=F$ in $\Omega$ and $u_0=f$ on $\partial\Omega$.
\end{theorem}

\begin{proof} See \cite{Jikov-1994} for the single equation case $(m=1)$.
The proof for the case $m>1$ is exactly the same.
\end{proof}

In Sections 3 and 4 we will use a compactness argument to establish the uniform  H\"older estimates
for local solutions of $\mathcal{L}_\varep (u_\varep) =\text{\rm div} (f) +F$.
This requires us to work with
 a class of operators that are obtained from $\mathcal{L}^A=-\text{div}
\big(A (x)\nabla \big)$
through translations and rotations of coordinates in $\rd$. Observe that if $ \mathcal{L}^A (u)=F$ and
$ x=Oy+z$ for some rotation $O=(O_{ij})$ and $z\in \rd$, then $\mathcal{L}^B (v)=G$, 
where  $v(y)=u(Oy+z)$, $B =(b_{ij}^{\alpha\beta} (y))$
with $b_{ij}^{\alpha\beta} (y) =a_{\ell k}^{\alpha\beta} (Oy +z) O_{\ell i} O_{kj}$,
and $G(y)= F(Oy +z)$.
Thus, for each $A=\big(a_{ij}^{\alpha\beta}\big)$ fixed, we shall consider the set of matrices,
\begin{equation}\label{class-A}
\aligned
\mathcal{A}
=\Big\{ B=\big( b_{ij}^{\alpha\beta} (y)): & \ b_{ij}^{\alpha\beta} (y) =a_{\ell k}^{\alpha\beta}(Oy+z) O_{\ell i} O_{k j}\\
&  \text{ for some rotation } O=(O_{ij})  \text{ and }  z\in \rd \Big\}.
\endaligned
\end{equation}
Note that if $B(y)= O^t A (Oy +z) O\in \mathcal{A}$, where $O^t$ denotes the transpose of $O$,
then 
the homogenized matrix $\widehat{B}
=O^t \widehat{A} O$.

The proof of Theorems \ref{holder-theorem} and \ref{boundary-holder-theorem}
 relies on the following extension of Theorem \ref{homo-theorem}.

\begin{theorem}\label{compactness-theorem}
Let $\Omega$ be a bounded Lipschitz domain in $\rd$ and $F\in H^{-1}(\Omega; \mathbb{R}^m)$.
Let $u_\ell \in H^1(\Omega; \mathbb{R}^{m})$ be a weak solution of
$-\text{\rm div} \big( A_\ell (x/\varep_\ell)\nabla u_\ell) =F$
in $\Omega$, where $\varep_\ell \to 0$ and $A_\ell \in \mathcal{A}$.
Suppose that $u_\ell \rightharpoonup u$ weakly in $H^1(\Omega; \mathbb{R}^m)$.
Then $u$ is a weak solution of $-\text{\rm div} \big( \widetilde{A}\nabla u\big) =F$
in $\Omega$, where $\widetilde{A}=O^t \widehat{A} O$
for some rotation $O$ in $\rd$.
\end{theorem}

\begin{proof}
Suppose that $A_\ell (y)=O_\ell^t A(O_\ell y+z_\ell) O_\ell$ 
for some rotations $O_\ell$ and $z_\ell\in \rd$.
By passing to a subsequence we may assume that $O_\ell  \to O$, as $\ell\to \infty$.
Since $A(y)$ is uniformly almost-periodic, $\{ A(y+z_\ell)\}_{\ell =1}^\infty$
is pre-compact in $C_b (\rd)$, the set of bounded continuous functions in $\rd$.
Thus, by passing to a subsequence, we may also assume that
$A(y+z_\ell)$ converges uniformly in $\rd$ to an almost-periodic matrix $B(y)$.
Consequently, we obtain $A_\ell (y)\to \widetilde{B}(y)=O^t B(Oy)O$ uniformly in $\rd$.
Note that $\widehat{\widetilde{B}}=O^t \widehat{B}O=O^t \widehat{A} O$.

Now, let $v_\ell\in H^1(\Omega; \mathbb{R}^m)$ be the weak solution of the Dirichlet problem:
$$
-\text{\rm div} \big(\widetilde{B} (x/\varep_\ell)\nabla v_\ell \big)=F \quad \text{ in }
\Omega \quad \text{ and } \quad v_\ell =u_\ell \quad \text{  on } \partial\Omega.
$$
Using $-\text{div} \big( A_\ell (x/\varep_\ell)\nabla (u_\ell -v_\ell)\big)
=\text{\rm div} \big( (A_\ell (x/\varep_\ell)-\widetilde{B} (x/\varep_\ell)) 
\nabla v_\ell \big)
$ in $\Omega$ and $u_\ell -v_\ell =0$ on $\partial\Omega$,
we may use the energy estimates to deduce that
$$
\aligned
\| u_\ell -v_\ell\|_{H^1(\Omega)}
&\le C\, \| A_\ell -\widetilde{B}\|_{L^\infty} \|\nabla v_\ell\|_{L^2 (\Omega)}\\
&\le C\,  \| A_\ell -\widetilde{B}\|_{L^\infty} \Big\{ \| u_\ell \|_{H^1(\Omega)}
+\| F\|_{H^{-1} (\Omega)} \Big\}.
\endaligned
$$
It follows that $u_\ell-v_\ell \to 0$ in $H^1(\Omega; \mathbb{R}^m)$, as $\ell\to \infty$.

Finally, since $v_\ell =v_\ell-u_\ell +u_\ell \rightharpoonup u$ weakly in $H^1(\Omega; \mathbb{R}^m)$,
it follows from Theorem \ref{homo-theorem} that
$\widetilde{B} (x/\varep_\ell)\nabla v_\ell 
\rightharpoonup \widetilde{A} \nabla u$
weakly in $H^1(\Omega; \mathbb{R}^{d\times m})$, where $\widetilde{A}=
\widehat{\widetilde{B}}=O^t \widehat{A} O$.
As a result, we obtain $-\text{\rm div} \big( \widetilde{A}\nabla u\big) =F$
in $\Omega$. This completes the proof.
\end{proof}


\section{Uniform interior H\"older estimates}
\setcounter{equation}{0}

The goal of this  and next sections is to establish uniform  interior and boundary H\"older estimates for solutions 
of $\mathcal{L}_\varep (u_\varep) =f +\text{\rm div} (g)$.
We will first use a compactness method to deal with the special case $\mathcal{L}_\varep (u_\varep)=0$.
The results are then used to establish size and H\"older estimates for fundamental solutions
and Green functions for $\mathcal{L}_\varep$.
The general case follows from the estimates for fundamental solutions and Green functions.

\begin{theorem}\label{holder-theorem}
Let $u_\varep \in H^1 (B(x_0, 2r); \mathbb{R}^m)$ be a weak solution of
$\text{\rm div} \big( A(x/\varep)\nabla u_\varep \big)=0$ in $B(x_0, 2r)$, for
some $x_0\in \rd$ and $r>0$. Let $\sigma\in (0,1)$.
Then
\begin{equation}\label{holder-estimate}
|u_\varep (x) -u_\varep (y)|
\le C_\sigma \, \left( \frac{|x-y|}{r}\right)^\sigma \left\{ \average_{B(x_0, 2r)}
|u_\varep|^2\right\}^{1/2}
\end{equation}
for any $x,y\in B(x_0, r)$,
where $C_\sigma$ depends only on $d$, $m$, $\sigma$ and $A$ (not on $\varep$, $x_0$, $r$).
\end{theorem}

Theorem \ref{holder-theorem} follows from Theorem \ref{compactness-theorem} by
a three-step compactness argument, similar to the periodic case in \cite{AL-1987}.

\begin{lemma}\label{step-1}
Let $0<\sigma<1$. Then there exist  constants $\varep_0>0$
and $\theta\in (0,1/4)$, depending only on $\sigma$ and $A$, such that
\begin{equation}\label{4.2-0}
\average_{B(y, \theta)} |u_\varep -\average_{B(y, \theta)} u_\varep|^2
\le \theta^{2\sigma} \qquad \text{ for any } 0<\varep<\varep_0,
\end{equation}
 whenever $u_\varep\in H^1(B(y, 1); \mathbb{R}^m)$ is
a weak solution of $\text{\rm div}\big(A(x/\varep)\nabla  u_\varep\big) =0$ in $B(y,1)$ for some $y\in \rd$, and
$$
\average_{B(y,1)} |u_\varep|^2\le 1.
$$
\end{lemma}

\begin{proof}
If $\text{\rm div}\big(A(x/\varep)\nabla u_\varep\big) =0$ in $B(y,1)$ and $v(x)=u_\varep (x+y)$,
then $\text{\rm div}\big(B(x/\varep)\nabla v\big) =0 $ in $B(0,1)$, where $B(x)= A(x+\varep^{-1} y) \in
\mathcal{A}$. As a result, it suffices to establish estimate (\ref{4.2-0}) for $y=0$ and for solutions $u_\varep$ of
$\text{\rm div}\big(B(x/\varep)\nabla u_\varep)=0$ in $B(0,1)$, where $B\in \mathcal{A}$.

To this end, we first note that if $w$ is a solution of a second-order elliptic system in $B(0,1/2)$
with constant coefficients satisfying the ellipticity condition (\ref{ellipticity-1}), then 
\begin{equation}\label{4.2-1}
\average_{B(0, \theta)} |w-\average_{B(0, \theta)} w|^2
\le C_0 \, \theta^2 \average_{B(0,1/2)} |w|^2 \quad \text{ for any } 0<\theta<1/4,
\end{equation}
where $C_0$ depends only on $d$, $m$ and $\mu$.
We now choose $\theta\in (0,1/4)$ so small that
\begin{equation}\label{4.2-3}
2^d \, C_0\,  \theta^2 <\theta^{2\sigma}.
\end{equation}
We claim that the estimate (\ref{4.2-0}) with $y=0$ holds for this $\theta$ and for some $\varep_0>0$, which 
depends only on $A$, whenever $u_\varep$ is a weak solution of $\text{\rm div} \big( B(x/\varep)\nabla u_\varep\big)
=0$ in $B(0,1)$ for some $B\in \mathcal{A}$.

Suppose this is not the case.
Then there exist $\{ \varep_\ell\}\subset \mathbb{R}_+$, $\{ B_\ell \}\subset \mathcal{A}$, and $\{ u_\ell\}\subset
H^1(B(0,1);\mathbb{R}^m)$ such that $\varep_\ell\to 0$, 
\begin{equation}\label{4.2-4}
\left\{
\aligned
& \text{\rm div} \big (B_\ell (x/\varep_\ell)\nabla u_\ell\big) =0 \quad \text{ in } B(0,1),\\
&\average_{B(0,1)} |u_\ell |^2 \le 1,
\endaligned
\right.
\end{equation}
and
\begin{equation}\label{4.2-5}
\average_{B(0,\theta)} |u_\ell -\average_{B(0, \theta)} u_\ell|^2 >\theta^{2\sigma}.
\end{equation}
Since $\{ u_\ell\}$ is bounded in $L^2(B(0,1);\mathbb{R}^m)$,
by Cacciopoli's inequality, $\{ u_\ell \}$ is bounded in $H^1(B(0,1/2); \mathbb{R}^m)$.
By passing to a subsequence we may assume that $u_\ell \rightharpoonup u$ weakly in $H^1(B(0,1/2); \mathbb{R}^m)$
and  in $L^2(B(0,1); \mathbb{R}^m)$.
It follows from Theorem \ref{compactness-theorem} that $u$ is a solution of
$\text{\rm div} \big(\widetilde{A} \nabla u)=0$ in $B(0,1/2)$, where $\widetilde{A}=
O^t \widehat{A} O$ for some rotation $O$ in $\rd$.
Since the matrix $ O^t\widehat{A} O$ satisfies the ellipticity condition 
(\ref{ellipticity-1}), estimate (\ref{4.2-1}) holds for $w=u$.
However, since $u_\ell\to u$ strongly in $L^2(B(0,1/2); \mathbb{R}^m)$, we may deduce from 
(\ref{4.2-5}) that
\begin{equation}\label{4.2-7}
\theta^{2\sigma}
\le \average_{B(0, \theta)} |u-\average_{B(0, \theta)} u |^2
\le C_0\, \theta^2 \average_{B(0, 1/2)} |u|^2 
\le 2^d \, C_0\, \theta^2 \average_{B(0,1)} |u|^2,
\end{equation}
where we have used (\ref{4.2-1}) for the second inequality. 

Finally, we note that the weak convergence of $u_\ell$ in $L^2(B(0,1); \mathbb{R}^m)$ and
the inequality in (\ref{4.2-4}) give
$$
\average_{B(0,1)} |u|^2\le 1.
$$ 
In view of (\ref{4.2-7}) we obtain $\theta^{2\sigma}\le 2^d \, C_0\, \theta^2$,
which contradicts (\ref{4.2-3}).
This completes the proof.
\end{proof}

\begin{lemma}\label{step-2}
Fix $0<\sigma<1$.
Let $\varep_0$ and $\theta$ be the constants given by Lemma \ref{step-1}. Let
$u_\varep \in H^1(B(y,1);\mathbb{R}^m)$ be a weak solution of $\text{\rm div} (A(x/\varep)\nabla u_\varep\big)
=0$ in $B(y,1)$ for some $y\in \rd$.
Then, if $0<\varep<\varep_0\, \theta^{k-1}$ for some $k\ge 1$, then
\begin{equation}\label{4.3-0}
\average_{B(y,\theta^k)}
|u_\varep -\average_{B(y, \theta^k)} u_\varep|^2
\le \theta^{2k\sigma} \average_{B(y,1)} |u_\varep|^2.
\end{equation}
\end{lemma}

\begin{proof}
The lemma is proved by an induction argument on $k$, using Lemma \ref{step-1} and
the rescaling property that
if $\mathcal{L}_\varep (u_\varep)=0$ in $B(y,1)$ and $v(x)=u_\varep (\theta^k x)$, then
$$
\mathcal{L}_{\frac{\varep}{\theta^k}} (v) =0 \quad \text{ in } B(\theta^{-k} y, \theta^{-k}).
$$
See \cite{AL-1987} for the periodic case.
\end{proof}

\begin{proof}[\bf Proof of Theorem \ref{holder-theorem}]
By rescaling we may assume that $r=1$.
Suppose that $u_\varep \in H^1(B(y,2);\mathbb{R}^m)$ and
$\text{\rm div} \big(A(x/\varep)\nabla u_\varep\big) =0$ in $B(y,2)$ for some $y\in \rd$.
We show that
\begin{equation}\label{4.4-1}
\average_{B(z,t)} |u_\varep -\average_{B(z,t)} u_\varep|^2
\le C \, t^{2\sigma} \average_{B(z,1)} |u_\varep|^2
\end{equation}
for any $0<t<\theta$ and $z\in B(y,1)$, where $\theta\in (0,1/4)$ is given by Lemma \ref{step-1}.
The  estimate (\ref{holder-estimate}) follows from (\ref{4.4-1})
by Campanato's characterization of H\"older spaces.

With Lemma \ref{step-2} as our disposal,
the proof of (\ref{4.4-1}) follows the same line of argument as in the periodic case.
We refer the reader to \cite{AL-1987} for details.
We point out that the classical local H\"older estimates for solutions
of elliptic systems in divergence form with continuous coefficients are needed  to handle the 
case $\varep\ge \theta\varep_0$ and $0<t<\theta$, as well as the case 
$0<\varep<\theta\varep_0$ and $0<t<\varep/\varep_0$.
\end{proof}

It follows from (\ref{holder-estimate}) and Cacciopoli's inequality that
\begin{equation}\label{holder-estimate-1-0}
\average_{B(y, t)} |\nabla u_\varep|^2
\le C_\sigma \left(\frac{t}{r}\right)^\sigma \average_{B(y,r)} |\nabla u_\varep|^2
\quad \text{ for any } 0<t<r,
\end{equation}
if $\text{\rm div} \big(A(x/\varep)\nabla u_\varep\big)=0$ in $B(y,r)$.
Since $A^*$ satisfies the same ellipticity and almost periodicity conditions as $A$,
estimate (\ref{holder-estimate-1}) also holds for solutions of 
$\text{\rm div} \big(A^*(x/\varep)\nabla u_\varep\big)=0$ in $B(y,r)$.
As a result, one may construct an $m\times m$ matrix of fundamental solutions
$\Gamma_\varep (x,y)=\big( \Gamma_\varep^{\alpha\beta} (x,y)\big)$ such that
for each $y\in \rd$, $\nabla_x \Gamma_\varep (x,y)$ is locally integrable and
\begin{equation}\label{representation}
\phi^\gamma (y)
=\int_{\rd} a_{ij}^{\alpha\beta} (x/\varep) \frac{\partial}{\partial x_j}
\Big\{ \Gamma_\varep^{\beta\gamma} (x,y)\Big\} \frac{\partial\phi^\alpha}{\partial x_i}\, dx
\end{equation}
for any $\phi=(\phi^\alpha)\in C_0^1 (\rd, \mathbb{R}^m)$ (see e.g. \cite{Hofmann-2007}).
Moreover, if $d\ge 3$, the matrix $\Gamma_\varep (x,y)$ satisfies 
\begin{equation} \label{size-estimate}
|\Gamma_\varep (x,y)|\le C\, |x-y|^{2-d}
\end{equation}
for any $x,y\in\rd$ and $x\neq y$, and
\begin{equation}\label{size-estimate-1}
\aligned
|\Gamma_\varep (x+h, y)-\Gamma_\varep (x,y)| &\le \frac{C_\sigma |h|^\sigma}{|x-y|^{d-2+\sigma}},\\
|\Gamma_\varep (x, y+h)-\Gamma_\varep (x,y)| & \le \frac{C_\sigma |h|^\sigma}{|x-y|^{d-2+\sigma}},
\endaligned
\end{equation}
where $x,y,h\in \rd$ and $0<|h|\le (1/2)|x-y|$.
Since $\mathcal{L}^*_\varep \big( \Gamma_\varep (x, \cdot)\big) =0$ in $\rd\setminus \{ x\}$,
using Cacciopopli's inequality and (\ref{size-estimate})-(\ref{size-estimate-1}), we obtain 
\begin{equation}\label{fs-estimate-2}
\left\{\average_{R\le |y-x|\le 2 R} |\nabla_y \Gamma_\varep (x, y)|^2\, dy \right\}^{1/2}
\le \frac{C}{R^{d-1}},
\end{equation}
and
\begin{equation}\label{fs-estimate-3}
\left\{ \average_{R\le |y-x_0|\le 2R}
|\nabla_y \big\{ \Gamma_\varep (x,y)-\Gamma_\varep (z,y)|^2\, dy \right\}^{1/2}
\le \frac{C\, |x-z|^\sigma}{R^{d-1+\sigma}},
\end{equation}
where $x,z\in B(x_0, r)$ and $R\ge 2r$.

\begin{theorem}\label{holder-theorem-1}
Let $u_\varep \in H^1 (B(x_0, 2r);\mathbb{R}^m)$ be a weak solution of
$$
-\text{\rm div} \big( A(x/\varep)\nabla u_\varep \big)=f  +\text{\rm div} (g)\quad \text{  in } 2B=B(x_0, 2r).
$$
Let $0<\sigma<1$. 
Then, for any $x,z\in B=B(x_0,r)$, 
\begin{equation}\label{holder-estimate-1}
\aligned
|u_\varep (x) -u_\varep (z)|
&\le C\,  |x-z|^\sigma
\bigg\{ r^{-\sigma} \left(\average_{2B} |u_\varep|^2\right)^{1/2}
+\sup_{\substack{ y\in B\\ 0<t<r}} t^{2-\sigma} \left(\average_{B(y,t)} |f|^2\right)^{1/2}\\
&\qquad\qquad\qquad\qquad\qquad
+\sup_{\substack{ y\in B\\ 0<t<r}} t^{1-\sigma} \left(\average_{B(y,t)} |g|^2\right)^{1/2} \bigg\},
\endaligned
\end{equation}
where $C$ depends only on $p$, $\sigma$ and $A$. In particular,
\begin{equation}\label{L-infty-estimate}
\aligned
\| u_\varep\|_{L^\infty(B )}
 \le C
\left(\average_{2B} |u_\varep|^2\right)^{1/2}
&+C\, r^\sigma \, \sup_{\substack{ y\in B\\ 0<t<r}} t^{2-\sigma} \left(\average_{B(y,t)} |f|^2\right)^{1/2}\\
&+C\, r^\sigma \sup_{\substack{ y\in B\\ 0<t<r}} t^{1-\sigma} \left(\average_{B(y,t)} |g|^2\right)^{1/2},
\endaligned
\end{equation}
where $C$ depends only on $p$, $\sigma$ and $A$.
\end{theorem}

\begin{proof} 
We first note that the $L^\infty$ estimate (\ref{L-infty-estimate}) follows easily from (\ref{holder-estimate-1}).
To see (\ref{holder-estimate-1}), we assume $d\ge 3$;
the case $d=2$ follows from the case $d=3$ by adding a dummy variable (the method of ascending).
We choose a cut-off function $\varphi\in C_0^\infty (B(x_0,7r/4))$
such that $0 \le \varphi\le1$,
$\varphi=1$ in $B(x_0, 3r/2)$, and $|\nabla\varphi|\le Cr^{-1}$. 
Since
$$
\mathcal{L}_\varep (u_\varep)
=f \varphi +\text{\rm div} (g\varphi) -g\nabla \varphi
-A(x/\varep)\nabla u_\varep \cdot \nabla \varphi
-\nabla \big\{ A(x/\varep) u_\varep \cdot \nabla \varphi\big\},
$$
we obtain that for $x\in B(x_0,r)$,
\begin{equation}\label{4.6-0}
\aligned
u_\varep (x) &=\int_{\rd} \Gamma_\varep (x,y) f(y)\varphi (y)\, dy
-\int_{\rd} \nabla_y \Gamma_\varep (x,y) g(y) \varphi (y)\, dy\\
&\quad  -\int_{\rd} \Gamma_\varep (x,y) g(y) \nabla \varphi (y)\, dy
-\int_{\rd} \Gamma_\varep (x,y) A(y/\varep) \nabla u_\varep(y) \cdot \nabla \varphi (y)\, dy\\
&\quad
+\int_{\rd}\nabla_y \Gamma_\varep (x,y)  A(y/\varep) u_\varep(y) \nabla \varphi (y)\, dy.
\endaligned
\end{equation}
 It follows that for any $x, z\in B(x_0,r)$,
\begin{equation}\label{4.6-1}
\aligned
|u_\varep (x)-u_\varep (z)|
\le C\int_{2B}  & |\Gamma_\varep (x,y)-\Gamma_\varep (z,y)|\, |f(y)|\, dy\\
&+C \int_{2B} |\nabla_y \big\{ \Gamma_\varep (x,y)-\Gamma_\varep (z,y)\big\} |\, |g(y)|\, dy\\
&+ C\int_{2B} |\Gamma_\varep (x,y)-\Gamma_\varep (z,y)|\, |g(y)|\, |\nabla \varphi (y)|\, dy\\
&+C \int_{2B} |\Gamma_\varep (x,y)-\Gamma_\varep (z,y)|\, |\nabla u_\varep (y)|\, |\nabla \varphi (y)|\, dy\\
& +C\int_{2B}|\nabla_y \Gamma_\varep (x,y)-\nabla_y \Gamma_\varep (z,y)|\, |u_\varep (y)|\, |\nabla \varphi (y)|\, dy,
\endaligned
\end{equation}
where $2B=B(x_0, 2r)$.
Since $|\nabla \varphi|=0$ in $ B(x_0, 3r/2)$ and $x,z\in B(x_0,r)$,
 the last three terms in the right hand side of 
(\ref{4.6-1}) may be handled easily, using estimate (\ref{size-estimate-1}), Cacciopoli's
inequality, and (\ref{fs-estimate-3}). They are bounded by 
$$
C_\sigma \left(\frac{|y-z|}{r}\right)^\sigma\left\{
\left(\average_{2B} |u_\varep|^2\right)^{1/2}
+r^2 \left(\average_{2B} |f|^2\right)^{1/2}
+r \left(\average_{2B} |g|^2\right)^{1/2} \right\},
$$
for any $\sigma\in (0,1) $.

Next, we use  (\ref{size-estimate}) and (\ref{size-estimate-1})
 to bound the first term in the right hand side of (\ref{4.6-1}) by
\begin{equation}\label{4.6-1-1}
C \int_{B(x,4s)} \frac{|f(y)|\, dy}{|x-y|^{d-2}}
+ C\int_{B(z,5s)} \frac{|f(y)|\, dy}{|z-y|^{d-2}}
+C s^{\sigma_1} \int_{2B\setminus B(x,4s)} \frac{|f(y)|\, dy}{|x-y|^{d-2+\sigma_1}},
\end{equation}
where $s=|x-z|$ and $\sigma_1 \in (\sigma, 1)$.
By decomposing $B(x,4s)$ as a union of sets $ \{ y: |y-x|\sim 2^j s \}$,
 it is not hard to verify that the first term in (\ref{4.6-1-1}) is bounded
 by
 $$
 C\, s^\sigma \sup_{\substack{ y\in B \\ 0<t<r}} t^{2-\sigma} 
 \left(\average_{B(y,t)} |f|^2\right)^{1/2}.
 $$
 The other two terms in (\ref{4.6-1-1}) may be handled in a similar manner.
 
Finally, the second term in the right hand side of (\ref{4.6-1})
is bounded by
\begin{equation}\label{4.6-1-2}
\aligned
&\int_{B(x,4s)} |\nabla_y \Gamma_\varep (x,y)|\, |g(y)|\, dy
+\int_{B(z,5s)} |\nabla_y \Gamma_\varep (z,y)|\, |g(y)|\, dy\\
&\qquad \qquad +\int_{2B\setminus B(x, 4s)}
|\nabla_y \big\{ \Gamma_\varep (x,y) -\Gamma_\varep (z,y)\big\}|\, |g(y)|\, dy,
\endaligned
\end{equation}
By decomposing  $2B\setminus B(x, 4s)$ as a union of sets
$\{ y: |y-x|\sim 2^j s\}$, and using H\"older inequality and (\ref{fs-estimate-3})
(with $\sigma$ replaced by some $\sigma_1 \in (\sigma, 1)$),
we may bound the third term in (\ref{4.6-1-2}) by
$$
C\, s^\sigma \sup_{\substack{ y\in B\\ 0<t<r}} t^{1-\sigma}
\left(\average_{B(y,t)} |g|^2\right)^{1/2}.
$$
The other two terms in (\ref{4.6-1-2}) may be handled in a similar manner.
This completes the proof.
  \end{proof}

\begin{remark}\label{remark-4.1}
{\rm
Suppose that $
-\text{\rm div} \big(A(x/\varep)\nabla u_\varep\big)=f
$ 
in $ 2B$ and  $f\in L^p(2B; \mathbb{R}^m)$ for some $p\ge 2$, where $2B=B(x_0,2r)$.
Assume $d\ge 3$.
Using (\ref{4.6-0}) and Cacciopoli's inequality, we may obtain that
 \begin{equation}\label{4.7-1}
|u_\varep (x)|\le C \int_{2B} \frac{|f(y)|}{|x-y|^{d-2}}\, dy
+C \left(\average_{2B}|u_\varep|^2\right)^{1/2}
+C r^2 \left(\average_{2B} |f|^2\right)^{1/2}
\end{equation}
for any $x\in B=B(x_0,r)$.
By the fractional integral estimates, this gives
\begin{equation}\label{4.7-2}
\left(\average_{B} |u_\varep|^q\right)^{1/q}
\le C \left(\average_{2B} |u_\varep|^2\right)^{1/2}
+C r^2 \left(\average_{2B} |f|^p\right)^{1/p},
\end{equation}
where $0<\frac{1}{p}-\frac{1}{q} \le \frac{2}{d}$.
}
\end{remark}



\section{Uniform boundary H\"older estimates and proof of Theorem \ref{main-theorem-3}}
\setcounter{equation}{0}

For $x_0\in \partial\Omega$ and $0<r<r_0=\text{diam}(\Omega)$, define
\begin{equation}\label{Omega}
\Omega_r (x_0)=B(x_0, r)\cap \Omega \quad \text{ and } \quad \Delta_r (x_0)=B(x_0, r)\cap \partial\Omega.
\end{equation}

\begin{theorem}\label{boundary-holder-theorem}
Let $\Omega$ be a bounded $C^{1, \eta}$ domain in $\rd$ for some $\eta>0$.
Let $u_\varep\in H^1(\Omega_r(x_0); \mathbb{R}^m)$ be a weak solution of
$\mathcal{L}_\varep (u_\varep) =0$ in $\Omega_r (x_0)$ and $u_\varep =0$ on $\Delta_r (x_0)$,
for some $x_0\in \partial\Omega$ and $0<r<r_0$.
Then, for any $0<\sigma<1$ and $x,y\in \Omega_{r/2} (x_0)$,
\begin{equation}\label{boundary-holder}
|u_\varep (x)-u_\varep (y)|
\le C\, \left(\frac{|x-y|}{r}\right)^\sigma \left(\average_{\Omega_r (x_0)} |u_\varep|^2\right)^{1/2},
\end{equation}
where $C$ depends only on $\sigma$, $A$ and $\Omega$.
\end{theorem}

Let $\phi: \mathbb{R}^{d-1} \to \mathbb{R}$ be a $C^{1, \eta}$ function
such that 
\begin{equation}\label{phi}
\phi (0)=0, \ \ 
\nabla \phi (0)=0, \text{ and }\ \  \| \nabla \phi\|_{C^{0, \eta} (\mathbb{R}^{d-1})}
\le M_0.
\end{equation}
Let 
\begin{equation}\label{D-I}
\aligned
D(r)=D(r, \phi) &=\big\{ (x^\prime, x_d)\in \rd: \ |x^\prime|< r \text{ and }  \phi (x^\prime)< x_d<
\phi (x^\prime) +10 (M_0 +1) r\big\},\\
I(r)= I (r, \phi) & = \big\{ (x^\prime, \phi (x^\prime))\in \rd: \ 
|x^\prime|<r \big\}.
\endaligned
\end{equation}
By translation and rotation Theorem \ref{boundary-holder-theorem} may be reduced to the following.

\begin{theorem}\label{b-h-theorem-1}
Let $u_\varep \in H^1 (D(r);\mathbb{R}^m)$ be a weak solution of $\text{\rm div} \big(B(x/\varep)\nabla u_\varep\big)=0$
in $D(r)$  and $u_\varep =0$ on $I(r)$, for some $r>0$ and $B\in \mathcal{A}$.
Then, for any $0<\sigma<1$ and $x,y\in D(r/2)$, 
\begin{equation}\label{b-h-estimate-1}
|u_\varep (x)-u_\varep (y)|
\le C\, \left(\frac{|x-y|}{r}\right)^\sigma \left(\average_{D_r} |u_\varep|^2\right)^{1/2},
\end{equation}
where $C$ depends only on $\sigma$, $A$ and $(\eta, M_0)$ in (\ref{phi}).
\end{theorem}

To prove Theorem \ref{b-h-theorem-1} we need a homogenization result 
for a sequence of matrices in the class $\mathcal{A}$ on a sequence of domains.

\begin{lemma}\label{sequence-homo-lemma}
Let $\{ B_\ell \}$ be a sequence of matrices in  $\mathcal{A}$.
Let $\{ \phi_\ell \}$ be a sequence of $C^{1, \eta}$ functions satisfying (\ref{phi}).
Suppose that $\text{\rm div} (B_\ell (x/\varep_\ell)\nabla u_\ell) =0$ in $D(r, \phi_\ell)$
and $u_\ell =0$ on $I (r, \phi_\ell)$ for some $r>0$, where $\varep_\ell \to 0$ and
$\| u_\ell \|_{H^1(D(r, \phi_\ell))} \le C$.
Then there exist subsequences of $\{ \phi_\ell\}$ and $\{ u_\ell\} $, which we still denote by
$\{ \phi_\ell \}$ and $\{ u_\ell\}$ respectively, and a function $\phi$ satisfying (\ref{phi}),
$u\in H^1(D(r, \phi); \mathbb{R}^m)$, and a constant matrix $\widetilde{B}$, such that
\begin{equation}\label{s-h-1}
\left\{
\aligned
& \phi_\ell \to \phi \text{ in } C^1 (|x^\prime|<r),\\
& u_\ell (x^\prime, x_d -\phi_\ell (x^\prime))
\rightharpoonup u (x^\prime, x_d-\phi(x^\prime)) \text{ weakly in } H^1(D(r, 0); \mathbb{R}^m),
\endaligned
\right.
\end{equation}
and
\begin{equation}\label{s-h-2}
\text{\rm div} \big( \widetilde{B} \nabla u\big) =0 
\quad \text{ in } D(r, \phi) \quad \text{ and } \quad u =0 \text{ on } I (r, \phi).
\end{equation}
Moreover, the matrix $\widetilde{B}$, which is given by $O^t \widehat{A} O$ for some
rotation $O$ in $\rd$, satisfies the ellipticity condition (\ref{ellipticity-1}).
\end{lemma}

\begin{proof}
Since $\|\nabla \phi_\ell \|_{C^{0, \eta}(\mathbb{R}^{d-1})}\le M_0$ and
$\| u_\ell \|_{H^1(D(r, \phi_\ell))} \le C$, (\ref{s-h-1}) follows by passing to subsequences.
Suppose that $B_\ell (y)= O_\ell^t A(O_\ell y +z_\ell) O_\ell$ for some rotation 
$O_\ell$ and $z_\ell\in \rd$.
By passing to a subsequence, we may assume that $O_\ell \to O$.
Since $u_\ell \to u$ weakly in $H^1(\Omega; \mathbb{R}^m)$ for any
$\Omega\subset\subset D(r, \phi)$,
it follows from Theorem \ref{compactness-theorem} that 
$\text{\rm div} \big(\widetilde{B} \nabla u) =0$ in $D(r, \phi)$,
where $\widetilde{B}=O^t \widehat{A} O$.
Finally, since $v_\ell (x^\prime, x_d)=u_\ell (x^\prime, x_d +\phi_\ell (x^\prime))
\rightharpoonup v(x^\prime, x_d+\phi (x^\prime))$ weakly in $H^1(D(r, 0))$ and
$ v_\ell =0$ on $I (r, 0)$, we may conclude that $v=0$ on $I(r, 0)$.
Hence, $u=0$ on $I(r, \phi)$.
\end{proof}

\begin{proof}[\bf Proof of Theorem \ref{b-h-theorem-1}]
With Lemma \ref{sequence-homo-lemma} at our disposal,
Theorem \ref{b-h-theorem-1} follows by the three-step compactness argument, as
in the periodic case. We refer the reader to \cite{AL-1987} for details.
\end{proof}

With interior and boundary H\"older estimates in Theorems \ref{holder-theorem} and \ref{boundary-holder-theorem},
one may construct an $m \times m$ matrix $ G_\varep (x,y) =\big(G^{\alpha\beta}_\varep (x,y)\big)$
of Green functions for $\mathcal{L}_\varep$ for
a bounded $C^{1, \eta}$ domain $\Omega$.
Moreover, if $d\ge 3$, 
\begin{equation}\label{Green-size}
|G_\varep (x,y)|\le C\,  |x-y|^{2-d}
\end{equation}
for any $x,y\in \Omega$, and
\begin{equation}\label{Green-holder}
|G_\varep (x,y)-G_\varep (z,y)|\le \frac{C_\sigma \, |x-z|^\sigma }{|x-y|^{d-2+\sigma}}
\end{equation}
for any $x,y,z\in \Omega$ with $|x-z|<(1/2)|x-y|$ and for any $0<\sigma<1$.
Since $G_\varep (\cdot, y)=0$ and $G_\varep (y, \cdot)=0$ on $\partial\Omega$,
one also has
\begin{equation}\label{Green-boundary}
|G_\varep (x, y)|\le   \frac{ C\,  \left[ \delta (x)\right]^{\sigma_1} \big[\delta (y) \big]^{\sigma_2} }
{|x-y|^{d-2+\sigma_1 +\sigma_2}}
\end{equation}
for any $x, y\in \Omega$ and any $0\le \sigma_1, \sigma_2<1$, where $\delta (x)=\text{dist} (x, \partial\Omega)$
and $C$ depends only on $A$, $\Omega$, $\sigma_1$ and $\sigma_2$.

\begin{theorem}\label{b-h-lemma-3}
Let $\Omega$ be a bounded $C^{1, \eta}$ domain in $\rd$ for some $\eta>0$.
Suppose that $\mathcal{L}_\varep (u_\varep) =F$ in $\Omega$
and $u_\varep =0$ on $\partial\Omega$. Then
\begin{equation}\label{b-h-3-0}
\| u_\varep \|_{C^{\alpha }(\overline{\Omega})}\le 
C_\alpha  \sup_{\substack{x\in \Omega\\ 0<r<r_0}}
r^{2-\alpha } \average_{\Omega(x, r)} |F|
\end{equation}
for any $0<\alpha<1$, where $r_0=\text{\rm diam} (\Omega)$ and
$C_\alpha $ depends only on $A$, $\Omega$, and $\alpha$.
\end{theorem}

\begin{proof} Since
$$
u_\varep (x)=\int_\Omega G_\varep (x, y) F(y)\, dy,
$$
it follows that for any $x, z\in \Omega$.
$$
|u_\varep (x)-u_\varep (z)|\le \int_\Omega |G_\varep (x,y)-G_\varep (z,y)|\, |F(y)|\, dy.
$$
Let $t=|x-z|$ and write $\Omega =\big[ \Omega\setminus B(x, 4t)\big] \cup \Omega (x, 4t)$.
To estimate the integral of $|G_\varep (x, y)-G_\varep (z,y)|\, |F(y)|$ over $\Omega (x, 4t)$,
we use the estimate (\ref{Green-size}).
This gives
$$
\aligned
\int_{\Omega (x, 4t)}    |G_\varep (x, y)-G_\varep (z,y)|\, |F(y)|\, dy
& \le C \int_{\Omega(x, 4t)} \frac{|F(y)|\, dy}{|x-y|^{d-2}}
+C \int_{\Omega(z, 5t)} \frac{|F(y)|\, dy}{|z-y|^{d-2}}\\
& \le C\,  t^\alpha\sup_{\substack{x\in \Omega\\ 0<r<r_0}}
r^{2-\alpha} \average_{\Omega(x, r)} |F|.
\endaligned
$$
For the integral over $\Omega\setminus B(x, 4t)$, we choose $\beta \in (\alpha, 1)$ and use
(\ref{Green-holder}) to obtain
$$
\aligned
\int_{\Omega \setminus  B(x, 4t)}   |G_\varep (x, y)-G_\varep (z, y)|\, |F(y)|\, dy
 &\le C\,  t^\beta \int_{\Omega\setminus B(x, 4t)}
\frac{ |F(y)|\, dy}{|x-y|^{d-2 +\beta}}\\
&\le C\,  t^\alpha\sup_{\substack{x\in \Omega\\ 0<r<r_0}}
r^{2-\alpha} \average_{\Omega(x, r)} |F|.
\endaligned
$$
Thus we have proved that $|u(x)-u(z)|/|x-z|^\alpha$ is bounded by the right hand side of 
(\ref{b-h-3-0}). The remaining estimate for $\| u_\varep\|_{L^\infty(\Omega)}$ is similar.
\end{proof}

\begin{theorem}\label{b-h-lemma-4}
Let $\Omega$ be a bounded $C^{1, \eta}$ domain in $\rd$ for some $\eta>0$.
Suppose that $\mathcal{L}_\varep (u_\varep) =\text{\rm div} (f)$ in $\Omega$
and $u_\varep =0$ on $\partial\Omega$. Then
\begin{equation}\label{b-h-4-0}
\| u_\varep \|_{C^{\alpha}(\overline{\Omega})}\le 
C_\alpha\,  \sup_{\substack{x\in \Omega\\ 0<r<r_0}}
r^{1-\alpha} \left( \average_{\Omega(x, r)} |f|^2\right)^{1/2}
\end{equation}
for any $0<\alpha<1$, where $r_0=\text{\rm diam}(\Omega)$ and
$C_\alpha $ depends only on $A$, $\Omega$, and $\alpha$.
\end{theorem}

\begin{proof}
The proof is similar to that of Theorem \ref{b-h-lemma-3}, using
$$
|u_\varep (x)-u_\varep (z)|
\le \int_\Omega |\nabla _y \big\{ G_\varep (x, y)-G_\varep (z, y)\big\}|\, |f (y)|\, dy.
$$
The lack of point-wise estimates for $\nabla_y G_\varep (x, y)$ is overcome by using the following  estimates:
\begin{equation}\label{b-h-4-1}
\aligned
& \int_{r\le |y-x|\le 2r} |\nabla_y G_\varep (x,y)|^2\, dy
\le \frac{C}{r^2} \int_{(r/2)\le |y-x|\le 3r} |G_\varep (x,y)|^2\, dy,\\
& \int_{R\le |y-x|\le 2R} 
|\nabla_y \big\{ G_\varep (x,y) -G_\varep (z,y)\big\}|^2\, dy\\
 &\qquad\qquad \qquad 
 \le \frac{C}{R^2}
\int_{(R/2)\le |y-x|\le 3R} 
| G_\varep (x,y) -G_\varep (z,y)|^2\, dy,
\endaligned
\end{equation}
where $|x-z|<(1/4)|x-y|$. Estimate (\ref{b-h-4-1}) follows from Cacciopoli's inequality.
We omit the rest of the proof.
\end{proof}

\begin{theorem}\label{b-h-theorem-10}
Let $\Omega$ be a bounded $C^{1, \eta}$ domain in $\rd$ for some $\eta>0$.
Suppose that $\mathcal{L}_\varep (u_\varep) =0$ in $\Omega$ and $u_\varep =g$ on $\partial\Omega$.
Then
\begin{equation}\label{b-h-10-0}
\| u_\varep\|_{C^\alpha (\overline{\Omega})} \le C_\alpha\, \| g\|_{C^{\alpha} (\partial\Omega)}
\end{equation}
for any $0<\alpha<1$, where $C_\alpha$ depends only on $A$, $\Omega$, and $\alpha$.
\end{theorem}

\begin{proof}
Without loss of generality we may assume that $\| g\|_{C^\alpha(\partial\Omega)} =1$.
Let $v$ be the harmonic function in $\Omega$ such that  $v\in C (\overline{\Omega})$
and $v=g$ on $\partial\Omega$.
It is well known that $\| v\|_{C^\alpha (\overline{\Omega})} \le C_\alpha\, \| g\|_{C^\alpha(\partial\Omega)}= C_\alpha$,
where $C_\alpha$ depends only on $\alpha$ and $\Omega$.
By interior estimates for harmonic functions, one also has
\begin{equation}\label{b-h-10-1}
|\nabla v(x)|\le C_\alpha \big[\delta (x)\big]^{\alpha -1}
\end{equation}
for any $x\in \Omega$.
Since $\mathcal{L}_\varep (u_\varep -v)=-\mathcal{L}_\varep (v)$ in $\Omega$ and
$u_\varep -v=0$ on $\partial\Omega$,
it follows that
$$
u_\varep (x) -v(x) =-\int_\Omega \nabla_y G_\varep (x,y) A(y/\varep)\nabla v (y)\, dy.
$$
This, together with (\ref{b-h-10-1}), gives
\begin{equation}\label{b-h-10-3}
|u_\varep (x) -v(x)|
\le C_\alpha\,  \int_\Omega
|\nabla_y G_\varep (x,y)|\, \big[ \delta (y)\big]^{\alpha -1}\, dy.
\end{equation}
We will show that
\begin{equation}\label{b-h-10-4}
\int_\Omega |\nabla_y G_\varep (x,y)|\, \big[ \delta (y)\big]^{\alpha -1}\, dy 
\le C_\alpha \, \big[\delta (x)\big]^{\alpha}\quad \text{ for any } x\in \Omega.
\end{equation}

Assume (\ref{b-h-10-4}) for a moment. 
Then
\begin{equation}\label{b-h-10-5}
|u_\varep (x) -v (x)|\le C_\alpha \big[ \delta (x)\big]^{\alpha} \quad \text{ for any } x\in \Omega.
\end{equation}
It follows that $\| u_\varep\|_{L^\infty (\Omega)} 
\le \| v\|_{L^\infty(\Omega)} + C \le C$.
Let $x, y\in \Omega$.
To show $| u_\varep (x)-u_\varep (y)|\le C\, |x-y|^\alpha$, we consider three cases:
(1) $|x-y|<(1/4) \delta (x)$; (2)  $|x-y|<(1/4)\delta (y)$; 
(3) $|x-y|\ge \max \big((1/4)\delta(x), (1/4)\delta (y)\big)$.
In the first case, since $\mathcal{L}_\varep (u_\varep)=0$ in $\Omega$, we may use the
interior H\"older estimates in Theorem \ref{holder-theorem} to obtain
$$
|u_\varep (x)-u_\varep (y)|
\le C_\alpha\,  |x-y|^\alpha \, \| u_\varep\|_{L^\infty(B(x, \delta(x)/2))} \le C_\alpha \, |x-y|^\alpha.
$$
The second case is handled in the same manner.
For the third case we use (\ref{b-h-10-5}) and H\"older estimates for $v$  to see that
$$
\aligned
|u_\varep (x)-u_\varep (y)
&\le |u_\varep (x)-v(x)| +|v(x) -v(y)| +|v(y)-u_\varep (y)|\\
&\le C\, \big[\delta (x)\big]^\alpha +C \, |x-y|^\alpha + C \, \big[\delta (y)\big]^\alpha\\
&\le C_\alpha\, |x-y|^\alpha.
\endaligned
$$

It remains to prove (\ref{b-h-10-4}).
To this end we fix $x\in \Omega$ and let $r=\delta (x)/2$.
We first note that
\begin{equation}\label{b-h-10-7}
\aligned
\int_{B(x,r)} |\nabla_y G_\varep (x,y)|\big[\delta (y)\big]^{\alpha -1}\, dy
& \le C\,  r^{\alpha-1} \int_{B(x,r)} |\nabla_y G_\varep (x,y)|\, dy\\
&\le C\,  r^\alpha,
\endaligned
\end{equation}
where the last inequality follows from the first estimate in (\ref{b-h-4-1})
by decomposing $B(x, r)\setminus \{ 0\}$ as
$\cup_{j=0}^\infty \big\{ B(x, 2^{-j} r)\setminus B(x, 2^{-j-1} r)\big\} $.
To estimate the integral on $\Omega\setminus B(x, r)$,
we observe that if $Q$ is a cube in $\rd$ with the property that
$3Q\subset \Omega\setminus \{ x\}$ and
$\ell (Q)\sim \text{dist} (Q, \partial\Omega)$, 
then
\begin{equation}\label{b-h-10-8}
\aligned
\int_Q  & |\nabla_y G_\varep (x, y)|\big[\delta (y)\big]^{\alpha-1}\, dy
\le C\, \big[\ell (Q)\big]^{\alpha-1} |Q| \left(\average_Q |\nabla_y G_\varep (x,y)|^2\, dy\right)^{1/2}\\
&\le C\, \big[\ell (Q)\big]^{\alpha-2} |Q| \left(\average_{2Q} | G_\varep (x,y)|^2\, dy\right)^{1/2}\\
&\le C \, r^{\alpha_1} \big[\ell (Q)\big]^{\alpha+\alpha_2-2} |Q| \left(\average_{2Q} 
\frac{dy}{|x-y|^{2(d-2 +\alpha_1 +\alpha_2)}}\right)^{1/2},
\endaligned
\end{equation}
where $\alpha_1, \alpha_2 \in (0,1)$.
We remark that Cacciopoli's inequality was used for the second inequality above,
while the estimate (\ref{Green-boundary}) was used for the third.
Since $3Q\subset \Omega\setminus \{ x\}$, we see that $|x-y|\sim |x-z|$
for any $y,z\in 2Q$. As a result, it follows from (\ref{b-h-10-8}) that
\begin{equation}\label{b-h-10-9}
\int_Q |\nabla_y G_\varep (x, y)|\big[\delta (y)\big]^{\alpha-1}\, dy
\le C\, r^{\alpha_1}
\int_Q \frac{\big[\delta (y)\big]^{\alpha +\alpha_2 -2}}{|x-y|^{d-2 +\alpha_1+\alpha_2}}\, dy.
\end{equation}
By decomposing $\Omega\setminus B(x, r)$ as a non-overlapping union of cubes $Q$
with the said property (a Whitney type decomposition of $\Omega$),
we obtain 
\begin{equation}\label{b-h-10-10}
\aligned
\int_{\Omega\setminus B(x, r)} |\nabla_y G_\varep (x, y)|\big[\delta (y)\big]^{\alpha-1}\, dy
& \le C\, r^{\alpha_1}
\int_\Omega \frac{\big[\delta (y)\big]^{\alpha +\alpha_2 -2}}{( |x-y|+r)^{d-2 +\alpha_1+\alpha_2}}\, dy\\
& \le C\, r^{\alpha_1}
\int_{\mathbb{R}_+^d} \frac{ y_d^{\alpha +\alpha_2 -2} dy}{(| r-y_d| + r +|y^\prime|)^{d-2 +\alpha_1 +\alpha_2}}.
\endaligned
\end{equation}
Finally, a direct computation shows that the integral on the right hand side of (\ref{b-h-10-10})
is bounded by $Cr^{\alpha-\alpha_1}$, provided that $\alpha_1>\alpha$ and $\alpha_2>1-\alpha$.
This completes the proof.
\end{proof}

\begin{proof}[\bf Proof of Theorem \ref{main-theorem-3}]
This follows from Theorems \ref{b-h-lemma-3}, \ref{b-h-lemma-4} and \ref{b-h-theorem-10}
by writing $u_\varep =u_\varep^{(1)} +u_\varep^{(2)} +u_\varep^{(3)}$, where
$u_\varep^{(1)}$, $u_\varep^{(2)}$, $u_\varep^{(3)}$
satisfy the conditions in Theorems \ref{b-h-lemma-3}, \ref{b-h-lemma-4}, 
\ref{b-h-theorem-10}, respectively.
\end{proof}



\section{Construction of approximate correctors}
\setcounter{equation}{0}

In this section we construct the approximate correctors $\chi_T =\big (\chi_{T, j}^\beta\big)
=\big(\chi_{T, j}^{\alpha\beta}\big)$
and obtain some preliminary estimates.

\begin{prop}\label{lemma-2.1}
Let $f\in L^2_{\loc} (\rd; \mathbb{R}^m)$ and $g=(g_1, \dots, g_d)\in L^2_{\loc} (\rd; \mathbb{R}^{d\times m})$.
Assume that 
$$
\sup_{x\in \rd} \int_{B(x, 1)} \big (|f|^2 +|g|^2\big)<\infty.
$$
Then, for $T>0$,
 there exists a unique $u\in H^1_{\loc} (\rd; \mathbb{R}^m)$ such that
\begin{equation}\label{2.1-1}
-\text{\rm div} \big(A(x)\nabla u\big)  +T^{-2} u = f +\text{\rm div} (g) \quad \text{ in } \rd
\end{equation}
and 
\begin{equation}\label{2.1-2}
\sup_{x\in \rd} \int_{B(x,1)}\big( |\nabla u|^2 +|u|^2\big) <\infty.
\end{equation}
Moreover, the solution $u$ satisfies the estimate
\begin{equation}\label{2.1-3}
\sup_{x\in \rd} \average_{B(x,T)}\big( |\nabla u|^2 +T^{-2} |u|^2\big)
\le C \sup_{x\in \rd} \average_{B(x,T)} \big( |g|^2 +T^2 |f|^2\big),
\end{equation}
where $C$ depends only on $d$, $m$ and $\mu$.
\end{prop}

\begin{proof}
By rescaling we may assume that $T=1$.
The proof of the existence and estimate (\ref{2.1-3}) may be found in \cite{Yurinskii-1990}.
It uses the fact that for $f\in L^2 (\rd;\mathbb{R}^m)$,
$g=(g_1, \dots, g_d)\in L^2(\rd; \mathbb{R}^{d\times m})$
 with compact support, there exists a constant 
$\lambda>0$, depending only on $d$, $m$ and $\mu$, such that the solution of
(\ref{2.1-1}) in $H^1(\rd;\mathbb{R}^m)$ satisfies
$$
\int_{\rd} e^{\lambda |x|} \big\{ |\nabla u|^2 +|u|^2\big\}\, dx
\le C
\int_{\rd} e^{\lambda |x|} \big\{ |f|^2 +|g|^2\big\}\, dx.
$$

For the uniqueness, assume that $u\in H^1_{\loc} (\rd; \mathbb{R}^m)$ satisfies (\ref{2.1-2}) and
$-\text{\rm div} (A(x)\nabla u) +u=0$ in $\rd$.
By Cacciopoli's inequality,
$$
\int_{B(0, R)} |\nabla u|^2 +\int_{B(0,R)} |u|^2 \le \frac{C}{R^2} \int_{B(0,2R)} |u|^2
$$
for any $R\ge 1$. It follows that
$$
\int_{B(0, R)} |u|^2\le \frac{C}{R^{2d}} \int_{B(0, 2^d R)} |u|^2
$$
for any $R\ge 1$.
However, the condition (\ref{2.1-2}) implies that $\int_{B(0, 2^d R)} |u|^2\le C_u  R^d$.
Consequently, we obtain $\int_{B(0,R)} |u|^2\le C_u R^{-d}$ for any $R\ge 1$ and thus  $u\equiv 0$ in $\rd$.
\end{proof}

\begin{remark}\label{remark-2.0}
{\rm
The solution $u$ of (\ref{2.1-1}), given by Proposition \ref{lemma-2.1},
in fact satisfies 
\begin{equation}\label{p-estimate}
\sup_{x\in \rd}  \left\{ \average_{B(x,T)} |\nabla u|^p\right\}^{1/p}
\le C \sup_{x\in \rd} \left\{ \average_{B(x,T)} |g|^p\right\}^{1/p}
+C\, \sup_{x\in \rd} \left\{\average_{B(x,T)} T^2 |f|^2\right\}^{1/2},
\end{equation}
\begin{equation}\label{q-estimate}
\sup_{x\in \rd}
\left\{ \average_{B(x,T)} T^{-q} |u|^q\right\}^{1/q}
\le C \sup_{x\in \rd} \left\{ \average_{B(x,T)} |g|^p\right\}^{1/p}
+C\, \sup_{x\in \rd} \left\{\average_{B(x,T)} T^2 |f|^2\right\}^{1/2}
\end{equation}
for some $p>2$, depending only on $d$, $m$ and $\mu$, where $(1/q)=(1/p)-(1/d)$ for $d\ge 3$.
If $d=2$, the left hand side of (\ref{q-estimate}) should be replaced by
$ T^{-1} \|  u\|_{L^\infty}$.

To see (\ref{p-estimate}), one uses the weak reverse H\"older estimate:  if $u$ is a weak solution of
$-\text{\rm div} \big(A(x)\nabla u) =f +\text{\rm div} (g)$ in $B_r=B(x_0, r)$, then
$$
\left\{ \average_{B_{r/2}} |\nabla u|^p \right\}^{1/p}
\le \frac{C}{r} \left\{ \average_{B_r} |u|^2\right\}^{1/2}
+ C \left\{ \average_{B_r} |g|^p\right\}^{1/p}
+C\, r \left\{\average_{B_r} |f|^2 \right\}^{1/2}
$$
for some $p>2$, depending only on $d$, $m$ and $\mu$ (see e.g. \cite{Giaquinta}).
Estimate (\ref{q-estimate}) follows from (\ref{p-estimate}) by Sobolev imbedding.
}
\end{remark}

Let $P_j^\beta (x) =x_j e^\beta$, where $1\le j\le d$, $1\le \beta\le m$, and 
$e^\beta =(0, \dots, 1, \dots, 0)$ with $1$ in the $\beta^{th}$ position.
For $T >0$,
the approximate  corrector is defined as $\chi_T=\big( \chi_{T, j}^{\alpha\beta}\big) $,
where, for each $1\le j\le d$ and $1\le \beta\le m$, 
$u=\chi_{T, j}^\beta=\big (\chi_{T, j}^{1\beta}, \dots, \chi_{T, j}^{m\beta}\big)$ is the weak solution of
\begin{equation}\label{corrector-equation}
-\text{\rm div} \big(A(x)\nabla u\big) + T^{-2} u=\text{\rm div} \big(A (x)\nabla P_j^\beta\big)  \quad \text{ in }\rd,
\end{equation}
given by Proposition \ref{lemma-2.1}
It follows from (\ref{2.1-3}) that
\begin{equation}\label{c-e-2.0}
\sup_{x\in \rd} \average_{B(x, T)} \big( |\nabla \chi_T|^2 +T^{-2} |\chi_T|^2\big) \le C,
\end{equation}
where $C$ depends only on $d$, $m$ and $\mu$. Clearly, this gives 
\begin{equation}\label{corrector-estimate-2.1}
\sup_{\substack {x\in \rd\\ L\ge T} } 
\average_{B(x, L)} \big( |\nabla \chi_T|^2 +T^{-2} |\chi_T|^2\big) \le C,
\end{equation}
where $C$ depends only on $d$, $m$ and $\mu$.

\begin{lemma}\label{lemma-2.0}
Let $x,y, z\in \rd$. Then
\begin{equation}\label{2.0-1}
\aligned
&\left\{ \average_{B(x, T)} |\nabla \chi_T (t+y) -\nabla\chi_T (t+z)|^2\, dt\right\}^{1/2}
\le C\, \| A(\cdot +y) -A(\cdot +z)\|_{L^\infty(\rd)},\\
&
T^{-1}
\left\{ \average_{B(x,T)} |\chi_T(t+y)-\chi_T (t+z)|^2\, dt\right\}^{1/2}
\le C\,  \| A(\cdot +y) -A(\cdot +z)\|_{L^\infty(\rd)},
\endaligned
\end{equation}
where $C$ depends only on $d$, $m$ and $\mu$.
\end{lemma}

\begin{proof} Fix $y,z\in \rd$ and  $1\le j\le d$, $1\le \beta\le m$.
Let $u(t)=\chi_{T, j}^\beta (t+y)$ and $v(t)=\chi_{T, j}^\beta (t+z)$. Then $w=u-v$ is a solution of
$$
\aligned
& -\text{\rm div} \big( A(t+y)\nabla w\big)
+T^{-2} w\\
&\qquad  =\text{\rm div} \big( [A(t+y)-A(t+z)]\nabla P_j^\beta\big)
+\text{\rm div} \big( [A(t+y)-A(t+z)] \nabla v\big).
\endaligned
$$
In view of Proposition \ref{lemma-2.1} we obtain
$$
\aligned
& \average_{B(x,T)} \big( |\nabla w|^2 +T^{-2} |w|^2\big)
\le C\,  \sup_{x\in \rd} \average_{B(x,T)} | A(t+y)-A(t+z)|^2\, dt\\
&\qquad\qquad\qquad\qquad\qquad\qquad \quad +C\, \sup_{x\in \rd} \average_{B(x,T)} 
|A(t+y)-A(t+z)|^2 |\nabla v|^2\, dt\\
&\le C \, \| A(\cdot +y)-A(\cdot +z)\|_{L^\infty}^2
+C\, \| A(\cdot +y)-A(\cdot +z)\|_{L^\infty}^2
\sup_{x\in \rd} \average_{B(x,T)} |\nabla v|^2\\
&\le C\,  \| A(\cdot +y)-A(\cdot +z)\|_{L^\infty}^2,
\endaligned
$$
where we have used (\ref{c-e-2.0}) in the last inequality.
This completes the proof.
\end{proof}

\begin{remark}\label{remark-2.3}
{\rm
 For $f\in L^2_{\loc}(\rd)$, define
\begin{equation}\label{W-norm}
\| f\|_{W^2}
=\limsup_{L\to \infty} \sup_{x\in \rd} \left\{ \average_{B(x,L)} |f|^2\right\}^{1/2}.
\end{equation}
Note that by (\ref{c-e-2.0}),
\begin{equation}\label{c-e-2.1}
\|\nabla \chi_T \|_{W^2} + T^{-1} \| \chi_T \|_{W^2} \le C,
\end{equation}
where $C$ depends only on $d$, $m$ and $\mu$. 
Moreover, by Lemma \ref{lemma-2.0}, for any $\tau\in \rd$,
\begin{equation}\label{c-e-2.2}
\aligned
& \| \nabla \chi_T (\cdot +\tau) -\nabla \chi_T (\cdot) \|_{W^2} 
+ T^{-1} \| \chi_T (\cdot +\tau)- \chi_T (\cdot)\|_{W^2}\\
&\qquad\qquad\qquad\qquad
 \le C\, \| A(\cdot +\tau) -A(\cdot)\|_{L^\infty}.
\endaligned
\end{equation}
Since $A$ is uniformly almost-periodic, for any $\varep>0$, 
the set 
$$
\left\{ \tau\in \rd:\ \| A(\cdot +\tau)-A(\cdot)\|_{L^\infty(\rd)} <\varep \right\}
$$
is relatively dense in $\rd$.
It follows that for any $\varep>0$,
the set of $\tau$ for which the left hand side of (\ref{c-e-2.2}) is less than $\varep$ is also relatively dense in $\rd$.
By \cite{Bohr} this implies that $\nabla \chi_T$ and $\chi_T$ are limits of sequences of trigonometric polynomials 
with respect to the semi-norm $\| \cdot\|_{W^2}$ in (\ref{W-norm}).
In particular, $\nabla\chi_T, \chi_T\in B^2(\rd)$ for any $T>0$.
}
\end{remark}

\begin{lemma}\label{lemma-2.1-1}
Let $u_T=\chi_{T, j}^\beta$ for some $T>0$, $1\le j\le d$ and $1\le \beta\le m$.
Then
\begin{equation}\label{mean-equation}
\big\langle a_{ik}^{\alpha\gamma} \, \frac{\partial u_T^\gamma}{\partial x_k} \,  \frac{\partial v^\alpha}{\partial x_i}\big\rangle
+T^{-2} \langle u_T\cdot v\rangle 
=-\big\langle a_{ij}^{\alpha\beta}\, \frac{\partial v^\alpha}{\partial x_i}\big\rangle ,
\end{equation}
where $v=(v^\alpha)\in H^1_\loc (\rd; \mathbb{R}^{m})$
and $v^\alpha,\, \nabla v^\alpha\in B^2(\rd)$.
\end{lemma}

\begin{proof}
For any $\phi=(\phi^\alpha)\in H^1(\rd;\mathbb{R}^m)$ with compact support, we have
\begin{equation}\label{2.1-1-1}
\int_{\rd} a_{ik}^{\alpha\gamma}\,  \frac{\partial u_T^\gamma}{\partial x_k}\cdot
\frac{\partial \phi^\alpha}{\partial x_i}
+\frac{1}{T^2} \int_{\rd} u_T\cdot \phi
=-\int_{\rd} a_{ij}^{\alpha\beta} \, \frac{\partial \phi^\alpha}{\partial x_i}.
\end{equation}
Let $v=(v^\alpha)\in H^1_\loc (\rd; \mathbb{R}^{m})$.
Suppose that $v^\alpha\in B^2(\rd)$ and $\nabla v^\alpha\in B^2(\rd)$.
Choose $\phi (x)=\varphi (\varep x) v(x) $ in (\ref{2.1-1-1}), where $\varphi\in C_0^\infty(\rd)$.
The desired result follows 
by a simple change of variables $x\to x/\varep$ in (\ref{2.1-1-1}),
 multiplying both sides of the equation by $\varep^d$,
  and finally letting $\varep\to 0$.
\end{proof}

Letting $v$ be a constant in (\ref{mean-equation}), we see that 
\begin{equation}\label{m-c-mean}
\langle \chi_{T,j}^\beta\rangle =0.
\end{equation}
By taking $v=\chi_{T, j}^\beta$, we obtain
\begin{equation}\label{m-c-mean-1}
\langle A\nabla\chi_{T, j}^\beta \cdot \nabla \chi_{T, j}^\beta\rangle 
+T^{-2} \langle  |\chi_{T, j}^\beta|^2\rangle
=-\langle A^*\nabla \chi_{j, T}^\beta\rangle,
\end{equation}
where $A^*$ denotes the adjoint of $A$.
This, in particular, implies that
$$
\langle |\nabla \chi_T|^2\rangle  +T^{-2} \langle |\chi_T|^2\rangle  \le C,
$$
where $C$ depends only on $d$, $m$ and $\mu$.

\begin{lemma}\label{lemma-2.4}
Let $\psi= \left( \psi_{ij}^{\alpha\beta}\right)$ be defined by (\ref{c-equation}).
Then, as $T\to \infty$,
\begin{equation}\label{2.4-0-0}
\frac{\partial}{\partial x_i} \left( \chi_{T, j}^{\alpha\beta} \right)
\rightharpoonup \psi_{ij}^{\alpha\beta} \quad \text{\rm weakly  in } B^2(\rd).
\end{equation}
\end{lemma}

\begin{proof}
Fix $1\le j\le d$ and $1\le \beta\le m$.
Let $\widetilde{\psi}_j^\beta=\left(\widetilde{\psi}_{ij}^{\alpha\beta}\right)\in B^2(\rd;\mathbb{R}^{dm})$
be the weak limit  in $B^2(\rd)$
of a subsequence $\nabla \chi_{T_\ell, j}^{\beta}$,
where $T_\ell\to \infty$.
Since $\nabla \chi_{T, j}^{\beta}\in V^2_{\rm pot}$, we see that $\widetilde{\psi}_j^\beta\in V^2_{\rm pot}$.
Moreover, since $T^{-2} \langle |\chi_T|^2\rangle \le C$,
it follows by letting $T\to \infty$ in (\ref{mean-equation}) that
$$
\big\langle a_{ik}^{\alpha\gamma}\, \widetilde{\psi}_{kj}^{\gamma\beta}\, \frac{\partial v^\alpha}{\partial x_i}\big\rangle 
=-\big\langle a_{ij}^{\alpha\beta} \, \frac{\partial v^\alpha}{\partial x_i} \big\rangle
$$
for any $v=(v^\alpha)\in {\rm Trig} (\rd; \mathbb{R}^m)$.
This implies that $\widetilde{\psi}_j^\beta$ is a solution of  (\ref{c-equation}).
 By the uniqueness of the solution  we obtain $\widetilde{\psi}_j^\beta=\psi_j^\beta$ and hence (\ref{2.4-0-0}).
\end{proof}

\begin{theorem}\label{theorem-2.5}
As $T\to \infty$, $T^{-2}\langle |\chi_T|^2\rangle \to 0$.
\end{theorem}

\begin{proof} 
Note that
$$
\aligned
\mu \langle |\psi-\nabla \chi_T|^2\rangle
& \le \big\langle
a_{ik}^{\alpha\gamma} \left\{ \psi_{kj}^{\gamma\beta} -\frac{\partial}{\partial x_k} \left(\chi_{T, j}^{\gamma\beta}\right)\right\}
\left\{ \psi_{ij}^{\alpha\beta} -\frac{\partial}{\partial x_i} \left( \chi_{T, j}^{\alpha\beta} \right)\right\}\big\rangle\\
&
=\langle a_{ik}^{\alpha\beta}\psi_{kj}^{\gamma\beta} \psi_{ij}^{\alpha\beta}\rangle
-\big\langle a_{ik}^{\alpha\gamma}\frac{\partial}{\partial x_k} \left(\chi_{T, j}^{\gamma\beta}\right) \psi_{ij}^{\alpha\beta}\big\rangle
-T^{-2} \langle |\chi_{T}|^2\rangle ,
\endaligned
$$
where we have used equations (\ref{c-equation}) and (\ref{mean-equation}).
In view of Lemma \ref{lemma-2.4} this implies that as $T\to \infty$,
$T^{-2}\langle |\chi_T|^2\rangle \to 0$, and
\begin{equation}\label{strong-convergence}
\| \psi -\nabla \chi_T \|_{B^2} \to 0.
\end{equation}
\end{proof}

\begin{remark}\label{remark-2.5}
{\rm 
For $T>0$, let
\begin{equation}\label{m-homo-coefficient}
\widehat{a}^{\alpha\beta}_{T, ij}
=\langle a_{ij}^{\alpha\beta}\rangle
+\big\langle a_{ik}^{\alpha\gamma} \frac{\partial }{\partial x_k} \left( \chi_{T, j}^{\gamma\beta}\right)\big\rangle
\end{equation}
be the approximate homogenized coefficients. Then
\begin{equation}\label{homo-coef-diff}
\aligned
|\widehat{a}_{ij}^{\alpha\beta}
-\widehat{a}_{T, ij}^{\alpha\beta}|
 &=|\big\langle 
 a_{ik}^{\alpha\gamma} \left\{ \psi_{kj}^{\gamma\beta}-\frac{\partial}{\partial x_k}
\left(\chi_{T,j}^{\gamma\beta}\right)\right\}\big\rangle |\\
&\le C \, \| \psi-\nabla \chi_T\|_{B^2}.
\endaligned
\end{equation}
}
\end{remark}


\section{Estimates of approximate correctors}

\setcounter{equation}{0}

In this section we will establish sharp estimates for approximate correctors $\chi_T$.
The proof relies on the uniform $L^\infty$ and H\"older estimates obtained in Section 3
for solutions of $\mathcal{L}_\varep (u_\varep) =f +\text{\rm div} (g)$.

\begin{lemma}\label{lemma-5.1}
For $T\ge 1$, 
\begin{equation}\label{5.1-0}
\| \chi_T\|_{L^\infty(\rd)}  \le C\, T,
\end{equation}
where $C$ is independent of $T$. Moreover, for any $0<\sigma <1$ and $|x-y|\le T$,
\begin{equation}\label{5.1-0-0}
|\chi_T (x)-\chi_T (y)|\le C_\sigma \, T^{1-\sigma} |x-y|^\sigma,
\end{equation}
where $C_\sigma$ depends only on $\sigma$ and $A$.
\end{lemma}

\begin{proof}
We consider the case $d\ge 3$. The 2-d case follows by the method of ascending.

Let $1\le j\le d$ and $1\le \beta\le m$. Fix $z\in \rd$ and consider the function 
\begin{equation}\label{5.1-2-1}
u(x)=\chi_{T, j}^\beta (x) +P_j^\beta (x-z).
\end{equation}
It follows from (\ref{c-e-2.0}) that
\begin{equation}\label{5.1-1}
\left\{ \average_{B(z, 4T)}
|u|^2\right\}^{1/2} \le C \, T.
\end{equation}
Since
\begin{equation}\label{5.1-2}
\text{\rm div} \big( A(x)\nabla u\big) =T^{-2} \chi_{T,j}^\beta \qquad \text{ in } \rd,
\end{equation}
we may apply the estimate (\ref{4.7-2}) repeatedly to show that
\begin{equation}\label{5.1-3}
\left\{ \average_{B(z, 2T)}
|u|^p\right\}^{1/p} \le C_p \, T
\end{equation}
for any $2<p<\infty$, where $C_p$ depends only on $p$ and $A$.
This, together with (\ref{L-infty-estimate}), gives
$$
\| u\|_{L^\infty(B(z,T))} \le C\, T.
$$
Hence, $|\chi_{T, j}^\beta (z)|\le C\, T$ for any $z\in \rd$.
Finally, estimate (\ref{5.1-0-0}) follows from (\ref{5.1-0}) and the H\"older estimate (\ref{holder-estimate-1}).
\end{proof}

\begin{lemma}\label{lemma-5.1-1}
Let $\sigma_1, \sigma_2 \in (0,1)$ and $2<p<\infty$.
Then, for any $1\le  r\le T$,
\begin{equation}\label{L-p-estimate}
\sup_{x\in \rd}
\left(\average_{B(x,r)} |\nabla \chi_T|^p\right)^{1/p}
\le C\, T^{\sigma_1}  \left(\frac{T}{r}\right)^{\sigma_2},
\end{equation}
where $C$ depends only on $p$, $\sigma_1$, $\sigma_2$, and $A$.
\end{lemma}

\begin{proof}
Let $u$ be the same as in the proof of Lemma \ref{lemma-5.1}.
By Cacciopoli's inequality,
$$
\average_{B(z,r)} |\nabla u|^2  \le Cr^{-2} \average_{B(z,2r)} |u-u(z)|^2 + C r^2 \| T^{-2}\chi_T\|^2_{L^\infty},
$$
where $0<r\le T$.
In view of (\ref{5.1-0}) and (\ref{5.1-0-0}), this gives
\begin{equation}\label{5.2-1}
\sup_{z\in \rd} \left(\average_{B(z,r)} |\nabla \chi_T|^2\right)^{1/2}
\le C_\sigma \left(\frac{T}{r} \right)^{\sigma}
\end{equation}
for any $\sigma \in (0,1)$ and $0< r\le T$.
Since $A$ is uniformly continuous in $\rd$, by the local $W^{1,p}$ estimates for elliptic systems
in divergence form, it follows from (\ref{5.1-2}) that
$$
\left(\average_{B(z,1)} |\nabla u|^p\right)^{1/p}
\le
C_p \left(\average_{B(z,2)} |\nabla u|^2\right)^{1/2}
+ C\, T^{-2} \| \chi_T\|_{L^\infty},
$$
for any $z\in \rd$ and $2<p<\infty$,
where $C_p$ depends only on $p$ and $A$.
This, together with (\ref{5.2-1}), yields
$$
\sup_{z\in \rd}
\left(\average_{B(z,1)} |\nabla \chi_T|^p \right)^{1/p}
\le C_{p, \sigma} \, T^{\sigma}
$$
for any $\sigma\in (0,1)$ and $p\in (2,\infty)$.
Consequently,  for any $1\le r \le T$ and $\sigma\in (0,1)$, 
\begin{equation}\label{L-p-estimate-1}
\sup_{z\in \rd}
\left(\average_{B(z,r)} |\nabla \chi_T|^p \right)^{1/p}
\le C_{p, \sigma} \, T^{\sigma}.
\end{equation}
The desired estimate (\ref{L-p-estimate}) now follows from 
(\ref{5.2-1}) and (\ref{L-p-estimate-1}) by a simple interpolation
of $L^p$ norms.
\end{proof}


\begin{theorem}\label{theorem-5.3}
Let $T\ge 1$. The approximate corrector $\chi_T$ is uniformly almost-periodic in $\rd$.
 Moreover, for any $y, z\in \rd$,
\begin{equation}\label{5.3-0}
\| \chi_T (\cdot +y) -\chi_T (\cdot +z)\|_{L^\infty(\rd)}
\le C\, T\, \| A(\cdot + y) -A (\cdot +z)\|_{L^\infty(\rd)},
\end{equation}
where $C$ is independent of $T$ and $y,z$.
\end{theorem}

\begin{proof} We assume $d\ge 3$.
The case $d=2$ follows from the case $d=3$ by the method of ascending.
Fix $y, z\in \rd$ and $1\le j\le d$, $1\le \beta\le m$.
Let $$
u (x) =\chi_{T, j}^\beta (x+y) -\chi_{T, j}^\beta (x+z).
$$
Note that
\begin{equation}\label{5.3-1}
\aligned
-\text{\rm div} \big( A(x+y)\nabla u\big)
=-T^{-2} u  &+\text{\rm div} \big[ \big(A(x+y)-A(x+z)\big) \nabla P_j^\beta\big]\\
&+\text{\rm div} \big[ \big(A(x+y)-A(x+z)\big)\nabla v\big],
\endaligned
\end{equation}
where $v (x)=\chi_{T,j}^\beta (x+z)$.
Let $B=B(x_0, T)$. As in the proof of Theorem \ref{holder-theorem-1}, we choose a cut-off function
$\varphi\in C_0^\infty(B(x_0, 7T/4))$ such that $\varphi=1$ in $B(x_0, 3T/2)$ and $|\nabla \varphi|\le C\,T^{-1}$.
Using the representation formula by fundamental solutions and (\ref{5.3-1}), we obtain, for any $x\in B$,
\begin{equation}\label{5.3-3}
\aligned
|u(x)| \le C \,T^{-2}  & \int_{2B} |\Gamma^y (x, t)| \, |u(t)|\, dt\\
& +C \| A(\cdot+y)-A(\cdot +z)\|_{L^\infty}\int_{2B} |\nabla_t \big( \Gamma^y (x,t)\varphi (t) \big)|\, dt\\
& +C \| A(\cdot+y)-A(\cdot +z)\|_{L^\infty}
\int_{2B}|\nabla v(t)|\,  |\nabla_t \big( \Gamma^y (x,t)\varphi (t)\big)|\, dt\\
& +C \, T \left(\average_{2B} |\nabla u|^2\right)^{1/2}
+C \left(\average_{2B} |u|^2\right)^{1/2},
\endaligned
\end{equation}
where we have used $\Gamma^y (x,t)=\Gamma (x+y, t+y)$ to denote the matrix of fundamental solutions for the
operator $-\text{div} \big( A(\cdot +y) \nabla \big)$ in $\rd$.
By Lemma \ref{lemma-2.0} the last two terms in the right hand side of (\ref{5.3-3})
are bounded by the right hand side of (\ref{5.3-0}).
Using the size estimate (\ref{size-estimate}) and Cacciopoli's inequality, it is also not hard to see that
the second term in the right hand side of (\ref{5.3-3}) is bounded by the right hand side of (\ref{5.3-0}).

To treat the third term in the right hand side of (\ref{5.3-3}), we note that
$$
\aligned
& \int_{2B} |\nabla v(t)| \, |\nabla_t \big(\Gamma^y (x,t)\varphi (t)\big)|\, dt\\
&\le C \sum_{\ell=0}^\infty
\left(\average_{|t-x|\sim 2^{-\ell} T} |\nabla v(t)|^2\, dt\right)^{1/2}
\left(\average_{|t-x|\sim 2^{-\ell} T} |\nabla_t \big(\Gamma^y (x,t) \varphi\big)|^2\, dt \right)^{1/2}
(2^{-\ell} T)^d\\
&\le C \sum_{\ell=0}^\infty
(2^\ell )^{\sigma} \cdot  (2^{-\ell} T)^{1-d} \cdot (2^{-\ell} T)^d\\
&\le C\, T,
\endaligned
$$
where $\sigma \in (0,1)$ and
we have used (\ref{5.2-1}) to estimate the integral involving $|\nabla v (t)|^2$
for the second inequality.
As a result, we have proved that for any $x\in B$,
\begin{equation}\label{5.3-5}
|u(x)| \le C\, T^{-2} \int_{2B} \frac{|u(t)|}{|x-t|^{d-2}}\, dt +
C\, T\, \| A(\cdot +y) -A(\cdot +z)\|_{L^\infty}.
\end{equation}
By the fractional integral estimates, this implies that
$$
\left(\average_B |u|^q\right)^{1/q}
\le C \left(\average_{2B} |u|^p \right)^{1/p} 
+C\, T\, \| A(\cdot +y) -A(\cdot +z)\|_{L^\infty},
$$
where $1<p<q\le \infty$ and $(1/p)-(1/q)< (2/d)$.
Since 
$$
 \left(\average_{2B} |u|^2 \right)^{1/2} \le  
C\, T\, \| A(\cdot +y) -A(\cdot +z)\|_{L^\infty}
$$
by Lemma \ref{lemma-2.0}, a simple iteration argument shows that
$$
\| u\|_{L^\infty(B)} 
\le C\,  T\, \| A(\cdot +y) -A(\cdot +z)\|_{L^\infty}.
$$
This completes the proof.
\end{proof}

\begin{remark}\label{remark-5.1}
{\rm 
Let $u(x)=\chi_{T} (x+y)-\chi_{T} (x+z)$, as in the proof of Theorem \ref{theorem-5.3}.
Then
\begin{equation}\label{remark-5.1-0}
|u(t)-u(s)|
\le C_\sigma \left(\frac{|t-s|}{T}\right)^\sigma T\, \| A(\cdot +y)-A(\cdot +z)\|_{L^\infty},
\end{equation}
for any  $\sigma \in (0,1)$ and $t,s\in \rd$, where $C_\sigma$ depends only on $\sigma$ and $A$.
This follows from (\ref{5.3-1}), (\ref{5.3-0}) and (\ref{holder-estimate-1}).
By Cacciopoli's inequality and (\ref{remark-5.1-0}) we may deduce that
\begin{equation}\label{remark-5.1-1}
\sup_{x\in \rd}
\left(\average_{B(x,r)} |\nabla u|^2\right)^{1/2}
\le C_\sigma \left(\frac{T}{r}\right)^\sigma \| A(\cdot +y)-A(\cdot +z)\|_{L^\infty}
\end{equation}
for any $\sigma \in (0,1)$.
}
\end{remark}

\begin{theorem}\label{theorem-5.4}
Let $T\ge 1$. Then
\begin{equation}\label{5.4-0}
T^{-1}\|\chi_T\|_{L^\infty(\rd)}
\le C_\sigma \left\{ \rho (R) +\left(\frac{R}{T}\right)^\sigma \right\}
\end{equation}
for any $R>0$ and $\sigma \in (0,1)$, where $C_\sigma$ depends only on $\sigma$ and $A$.
In particular, $T^{-1} \|\chi_T\|_{L^\infty(\rd)} \to 0$, as $T\to \infty$.
\end{theorem}

\begin{proof}
Let $y, z\in \rd$. Suppose $|z|\le R$.
Then
$$
\aligned
|\chi_T (y) -\chi_T (0)|
&\le |\chi_T (y) -\chi_T (z)| +|\chi_T (z)-\chi_T (0)|\\
&\le C\, T\, \| A(\cdot +y) -A(\cdot +z)\|_{L^\infty(\rd)}
+ C_\sigma \, T^{1-\sigma} R^\sigma,
\endaligned
$$
where we have used Theorem \ref{theorem-5.3} and Lemma \ref{lemma-5.1}.
It follows that
\begin{equation}\label{5.4-1}
\sup_{y\in \rd}
 T^{-1} |\chi_T (y) -\chi_T (0)|
 \le C\,  \rho(R) + C_\sigma \left(\frac{R}{T}\right)^\sigma
 \end{equation}
 for any $R>0$.
 
 Finally, we observe that
 $$
 \aligned
 |\chi_T (0)| & \le \left| \average_{B(0, L)} \big\{ \chi_T (y) -\chi_T (0)\big\}\, dy \right|
 +\left|\average_{B(0,L)}  \chi_T (y)\, dy \right|\\
 &\le \sup_{y\in \rd} |\chi_T (y) -\chi_T (0)| + 
\left|\average_{B(0,L)} \chi_T (y)\, dy \right|.
\endaligned
$$
Since $\langle \chi_T\rangle =0$, we may let $L\to \infty$  in the estimate above to obtain 
$$
|\chi_T (0)|\le 
\sup_{y\in \rd} |\chi_T (y) -\chi_T (0)|.
$$
This, together with (\ref{5.4-1}), yields the estimate (\ref{5.4-0}).
\end{proof}

For $T\ge 1$ and $\sigma>0$, define
\begin{equation}\label{Phi}
\Theta_\sigma (T) =\inf_{0<R\le T} 
\left\{ \rho (R) +\left(\frac{R}{T}\right)^\sigma\right\}.
\end{equation}
Note that $\Theta_\sigma (T)$ is a decreasing and continuous function of $T$ 
and $\Theta_\sigma (T)\to 0$ as $T\to \infty$.
It follows from Theorem \ref{theorem-5.4} that
\begin{equation}\label{5.5-1}
T^{-1} \| \chi_T\|_{L^\infty(\rd)} \le C_\sigma \, \Theta_\sigma (T)  \quad \text{ for any } T\ge 1,
\end{equation}
where $\sigma \in (0,1)$.
By taking $R=T^\alpha$ for some $\alpha\in (0,1)$ in (\ref{Phi}), we see that
\begin{equation}\label{5.5-2}
\Theta_\sigma (T) \le \rho (T^\alpha) + T^{-\sigma (1-\alpha)}.
\end{equation}
This, in particular, implies that
$$
\text{ if } \int_1^\infty \frac{\rho(r)}{r}\, dr<\infty, \text{ then } \int_1^\infty \frac{\Theta_\sigma (r)}{r}\, dr<\infty.
$$

\begin{theorem}\label{theorem-5.6}
Let $T\ge 1$. Then
\begin{equation}\label{5.6-0}
\left( \big\langle |\psi-\nabla \chi_T|^2\big\rangle \right)^{1/2}
\le C_\sigma \int_{T/2}^\infty \frac{\Theta_\sigma (r)}{r}\, dr
\end{equation}
for  $\sigma \in (0,1)$, where $C_\sigma$ depends only on $\sigma$ and $A$.
\end{theorem}

\begin{proof}
Fix $1\le j\le d$ and $1\le \beta\le m$.
Let $u=\chi_{T, j}^\beta$, $v=\chi_{2T, j}^\beta$, and $w=u-v$.
It follows from Lemma \ref{lemma-2.1-1} that
$$
\langle A\nabla w\cdot \nabla \varphi\rangle 
=\frac{1}{4T^2} \langle v\cdot \varphi\rangle  -\frac{1}{T^2} \langle u\cdot \varphi\rangle 
$$
for any $\varphi\in H^1_{\loc} (\rd, \mathbb{R}^m)$ with $\varphi, \nabla \varphi \in B^2(\rd)$.
By taking $\varphi=w$, we obtain 
\begin{equation}\label{5.6-1}
\aligned
\langle |\nabla w|^2\rangle 
 & \le C\, T^{-2}\left\{ \langle  |u|^2\rangle   +\langle |v|^2\rangle \right\}\\
 & \le C_\sigma \big\{ \Theta_\sigma (T) +\Theta_\sigma (2T)\big\}^2,
 \endaligned
\end{equation}
where we have used (\ref{5.5-1}) for the second inequality.
Hence, we have proved that
$$
\left( \big\langle |\nabla\chi_T -\nabla\chi_{2T}|^2\big\rangle \right)^{1/2}
\le C_\sigma \int_{T/2}^T \frac{\Theta_\sigma (r)}{r}\, dr,
$$
where we have used the fact that $\Theta_\sigma (r)$ is decreasing.
Consequently, 
\begin{equation}\label{5.6-3}
\sum_{\ell=0}^\infty \left(\langle |\nabla \chi_{2^\ell T} -\nabla\chi_{2^{\ell+1} T} |^2 \rangle \right)^{1/2}
\le C_\sigma \int_{T/2}^\infty \frac{\Theta_\sigma (r)}{r}\, dr.
\end{equation}
Recall that by (\ref{strong-convergence}),  $\langle |\psi-\nabla\chi_T|^2\rangle \to 0$ as $T\to \infty$.
The estimate (\ref{5.6-0}) now follows from (\ref{5.6-3}).
\end{proof}

\begin{remark}\label{remark-5.4}
{\rm
Suppose that there exist $C>0$ and $\tau>0$ such that
\begin{equation}\label{r-5.4-1}
\rho(R)\le {C}/{R^{\tau}}  \qquad \text{ for } R\ge 1.
\end{equation}
By taking $R=T^{\frac{\sigma}{\tau +\sigma}}$ in (\ref{5.4-0}), we obtain
$$
T^{-1} \| \chi_T\|_{L^\infty} \le C\, \Theta_\sigma (T)\le C\, T^{-\frac{\tau \sigma}{\tau +\sigma}}.
$$
Since $\sigma\in (0,1)$ is arbitrary, this shows that
\begin{equation}\label{r-5.4-2}
T^{-1}\| \chi_T\|_{L^\infty}
\le {C_\delta}\, {T^{-\frac{\tau}{\tau +1} +\delta}}
\end{equation}
for any $\delta\in (0,1)$, where $C_\delta$ depends only on $\delta$ and $A$.
Under the condition (\ref{r-5.4-1}), by Theorem \ref{theorem-5.6}, we also obtain
\begin{equation}\label{r-5.4-3}
\left(\langle |\psi-\nabla \chi_T|^2\rangle \right)^{1/2}
\le 
{C_\delta}\, {T^{-\frac{\tau}{\tau +1} +\delta}}
\end{equation}
for any $\delta\in (0,1)$.
}
\end{remark}



\section{Convergence rates}
\setcounter{equation}{0}

In this section we give the proof of Theorems \ref{main-theorem-1} and \ref{main-theorem-2}. 

\begin{lemma}\label{lemma-6.1}
Let $h\in L^2_{\loc} (\rd)$ and $T>0$.
Suppose that there exists $\sigma \in (0,1)$ such that
\begin{equation}\label{6.1-0}
\sup_{x\in \rd} \left(\average_{B(x,r)} |h|^2\right)^{1/2} \le \left(\frac{T}{r}\right)^{1-\sigma}
 \quad \text{ for any } 0<r\le T.
\end{equation}
Let $u \in  H^1_{\loc} (\rd)$ be the solution of
\begin{equation}\label{6.1-1}
-\Delta u + T^{-2} u = h \quad \text{ in } \rd,
\end{equation}
given by Proposition \ref{lemma-2.1}. Then
\begin{equation}\label{6.1-2}
\| u\|_{L^\infty} \le C\, T^2, \quad \| \nabla u\|_{L^\infty} \le C\, T,
\end{equation}
and
\begin{equation}\label{6.1-3}
|\nabla u(x)-\nabla u(y)|\le C\, T^{1-\sigma} |x-y|^\sigma\quad \text{ for any } x,y\in \rd,
\end{equation}
where $C$ depends only on $d$ and $\sigma$.
Furthermore, $ u\in H^2_{\loc} (\rd)$ and
\begin{equation}\label{6.1-3-3}
\sup_{x\in \rd} \left(\average_{B(x,T)} |\nabla^2 u|^2\right)^{1/2} \le C.
\end{equation}
\end{lemma}

\begin{proof}
By rescaling we may assume $T=1$.
It follows from Proposition \ref{lemma-2.1}  and (\ref{6.1-0}) that
\begin{equation}\label{6.1-4}
\sup_{x\in \rd} \left(\average_{B(x, 1)} |u|^2\right)^{1/2}\le C\,  \quad \text{ and } \quad 
\sup_{x\in \rd} \left(\average_{B(x, 1)} |\nabla u|^2\right)^{1/2}\le C\, ,
\end{equation}
where $C$ depends only on $d$.
Fix $x_0\in \rd$ and let $\phi\in C_0^\infty( B(x_0,2))$ be a cut-off function 
such that $\phi=1$ in $B(x_0,1)$.
By representing 
$u\phi$ as an integral and using the fundamental solution for $-\Delta$,
the desired estimates follow from (\ref{6.1-0}) by a standard procedure.
We leave the details to the reader.
\end{proof}

Under additional almost periodicity conditions on $h$, 
the next lemma gives much sharper estimates for the solution $u$ of (\ref{6.1-1}).

\begin{lemma}\label{lemma-6.2}
Let $h\in L^2_{\loc} (\rd)$ and $T>0$.
Suppose that there exists $\sigma \in (0,1)$ such that
\begin{equation}\label{6.2-0}
\aligned
\sup_{x\in \rd} \left(\average_{B(x,r)} |h|^2\right)^{1/2}  &\le C_0 \left(\frac{T}{r}\right)^{1-\sigma},\\
\sup_{x\in \rd} \left(\average_{B(x,r)} |h((t+y)
-h(t+z)|^2\, dt \right)^{1/2}  &\le C_0 \left(\frac{T}{r}\right)^{1-\sigma}
\| A(\cdot+y)-A(\cdot +z)\|_{L^\infty}
\endaligned
\end{equation}
for any $0<r\le T$ and $y,z\in \rd$.
Let $u \in  H^1_{\loc} (\rd)$ be the solution of (\ref{6.1-1}),
given by Proposition \ref{lemma-2.1}.
Then 
\begin{equation}\label{6.2-1}
\aligned
T^{-2}\, \| u\|_{L^\infty}  &\le C\,  \, \Theta_1 (T) + \, |\langle h \rangle |,\\
T^{-1}\, \|\nabla u\|_{L^\infty} & \le C\,  \Theta_\sigma (T),
\endaligned
\end{equation}
where $\Theta_\sigma (T)$ is defined by (\ref{Phi}) and
$C$ depends at most on $d$, $\sigma$ and $C_0$.
\end{lemma}

\begin{proof}
By applying Lemma \ref{lemma-6.1} to the function
$$
\big( u(x+y)-u(x+z)\big)/\big(C_0\| A(\cdot +y)-A(\cdot +z)\|_{L^\infty}\big),
$$
 with $y,z$ fixed, we obtain 
 \begin{equation}\label{6.2-3}
 \aligned
 \| u(\cdot +y)-u(\cdot +z)\|_{L^\infty}
 &\le C\, T^2 \, \| A(\cdot +y) -A(\cdot +z)\|_{L^\infty},\\
  \| \nabla u(\cdot +y)-\nabla u(\cdot +z)\|_{L^\infty}
 &\le C\, T \, \| A(\cdot +y) -A(\cdot +z)\|_{L^\infty},
 \endaligned
 \end{equation}
 where $C$ depends only on $d$, $C_0$ and $\sigma$.
 This shows that $u$ and $\nabla u$ are uniformly almost-periodic.
 In particular, $u$ and $\nabla u$ have mean values and $\langle \nabla u\rangle =0$.
 Also, note that condition (\ref{6.2-0}) implies that $h\in B^2(\rd)$
 and hence has the mean value $\langle h \rangle $.
 It is easy to deduce from the equation (\ref{6.1-1}) that
 $\langle u \rangle =T^2 \langle h \rangle $.

Note that for any $y\in \rd$ and $z\in \rd$ with $|z|\le R\le T$,
$$
\aligned
T^{-2} |u(y)- u(0)| & \le  T^{-2} |u(y)-u(z)|+ T^{-2} |u(z)-u(0)|\\
& \le C \, \| A(\cdot +y)- A(\cdot +z)\|_{L^\infty} +C\, T^{-1} R,
\endaligned
$$
where we have used (\ref{6.2-3}) and $\|\nabla u\|_{L^\infty} \le C\, T$ for the second inequality.
It follows from the definition of $\rho(R)$ that
$$
\sup_{y\in \rd} T^{-2} | u(y)- u(0)|
\le C\, \Big\{  \rho (R) + T^{-1} R \Big\}  \quad \text{ for any } 0<R\le T.
$$
By the definition of $\Theta_1$, this gives
\begin{equation}\label{6.2-5}
\sup_{y\in \rd} T^{-2} | u(y)- u(0)| \le C\, \Theta_1 (T).
\end{equation}
Using
$$
|T^{-2} u(0)|\le  T^{-2} \left|\average_{B(0, L)} \left\{ u(y) -u(0)\right\}\, dy \right|
+\left|\average_{B(0,L)} u(x) \right|
$$
for any $L>0$ and (\ref{6.2-5}), we see that by letting $L\to \infty$,
\begin{equation}\label{6.2-7}
|T^{-2} u(0)|\le C\, \Theta_1 (T) +T^{-2} |\langle u \rangle |
=C\, \Theta_1 (T) +|\langle h \rangle |.
\end{equation}
The first inequality in (\ref{6.2-1}) now follows from (\ref{6.2-5}) and (\ref{6.2-7}).

Finally, we point out that the second inequality in (\ref{6.2-1}) follows in the same manner, using (\ref{6.2-3})
and (\ref{6.1-3}) as well as the fact that the mean value of $\nabla u$ is zero.
\end{proof}

We are now ready to estimate the rates of convergence of $u_\varep$ to $u_0$.

\begin{theorem}\label{theorem-6.1}
Let $u_\varep$ $(\varep\ge 0)$ be the weak solution of $\mathcal{L}_\varep (u_\varep)=F$ in $\Omega$
and $u_\varep =g$ on $\partial\Omega$.
Suppose that $u_0\in W^{2, 2}(\Omega)$.
Let
\begin{equation}\label{6.3-1}
w_\varep (x)=u_\varep (x) -u_0 (x) -\varep \chi_{T, j} (x/\varep)\, \frac{\partial u_0}{\partial x_j}
+ v_\varep,
\end{equation}
where $T= \varep^{-1}$ and $v_\varep\in H^1(\Omega;\mathbb{R}^m)$ is the weak solution
of the Dirichlet problem:
\begin{equation}\label{6.3-1-1}
\mathcal{L}_\varep (v_\varep)=0 \quad \text{ in } \Omega \quad \text{ and } \quad 
 v_\varep =\varep \chi_{T, j} (x/\varep) \frac{\partial u_0}{\partial x_j}
 \quad \text{ on } \partial\Omega.
 \end{equation}
 Then
   \begin{equation}\label{6.3-1-1-1}
 \| w_\varep\|_{H^1(\Omega)}
 \le C_\sigma\, 
\Big\{ \Theta_\sigma (T) +
\langle |\psi-\nabla \chi_T|\rangle \Big\}
 \| u_0\|_{W^{2,2}(\Omega)}
 \end{equation}
 for any $\sigma \in (0,1)$, where $C_\sigma$ depends only on $\sigma$, $A$ and $\Omega$.
 \end{theorem}
 
 \begin{proof}
 With loss of generality we may assume that 
\begin{equation}\label{6.3-0}
\| u_0 \|_{W^{2, 2}(\Omega)} = 1.
\end{equation}
A direct computation shows that
\begin{equation}\label{6.3-2}
\mathcal{L}_\varep (w_\varep)
=-\text{\rm div} \big( B_T (x/\varep)\nabla u_0\big)
+\varep\,  \text{\rm div} \left\{ A(x/\varep) \chi_T (x/\varep) \nabla^2 u_0\right\},
\end{equation}
where $B_T (y)= \big(b_{T, ij}^{\alpha\beta} (y) \big)$ is given by
\begin{equation}\label{6.3-3}
b_{T, ij}^{\alpha\beta} (y)
=\widehat{a}^{\alpha\beta}_{ij}-a_{ij}^{\alpha\beta} (y)
-a_{ik}^{\alpha\gamma} (y)\, \frac{\partial}{\partial y_k}
\left\{\chi_{T, j}^{\gamma\beta} (y) \right\}.
\end{equation}
Since $w_\varep \in H^1_0(\Omega; \mathbb{R}^m)$,
it follows from (\ref{6.3-2}) that
\begin{equation}\label{6.3-5}
\aligned
c\int_\Omega |\nabla w_\varep|^2\, dx
&\le \left|\int_\Omega 
\text{\rm div} \big( B_T (x/\varep)\nabla u_0\big)\cdot w_\varep\, dx \right|
+\int_\Omega
|\varep\chi_T (x/\varep)|\, |\nabla^2 u_0| |\nabla w_\varep |\, dx\\
&= I_1 +I_2.
\endaligned
\end{equation}
It suffices to show that
\begin{equation}\label{6.3-5-5}
I_1 +I_2
\le C_\sigma \, \Big\{ \Theta_\sigma (T) +\langle |\psi-\nabla \chi_T|\rangle  \Big\} \| w_\varep\|_{H^1(\Omega)}
\end{equation}
for any $\sigma \in (0,1)$.

First, it is easy to see that
\begin{equation}\label{6.3-6}
I_2 \le C\, \varep\, \|\chi_T\|_{L^\infty} \|\nabla w_\varep\|_{L^2(\Omega)}
\le C\, \Theta_\sigma (T) \| \nabla w_\varep\|_{L^2(\Omega)}
\end{equation}
for any $\sigma \in (0,1)$, where we have used (\ref{6.3-0}) and (\ref{5.5-1}).

Next, to estimate $I_1$, we let $h(y)=h_T (y)=B_T (y)-\langle B_T\rangle $ and solve the equation (\ref{6.1-1}).
More precisely, let $h=\big(h_{ij}^{\alpha\beta}\big)$ and $f=(f_{ij}^{\alpha\beta})$, where $f_{ij}^{\alpha\beta}
\in H^2_{\loc} (\rd)$ sovles
\begin{equation}\label{6.3-7}
-\Delta f_{ij}^{\alpha\beta} + T^{-2} f_{ij}^{\alpha\beta} =h_{ij}^{\alpha\beta} \quad \text{ in } \rd.
\end{equation}
By (\ref{5.2-1}) and (\ref{remark-5.1-1}), the function $h$ satisfies the condition (\ref{6.2-0}) for any $\sigma \in (0,1)$.
Since $\langle h \rangle =0$, it follows from Lemma \ref{lemma-6.2} that
\begin{equation}\label{6.3-8}
\aligned
T^{-2}\, \| f\|_{L^\infty}  &\le C\,  \Theta_1 (T),\\
T^{-1}\, \|\nabla f \|_{L^\infty} & \le C \, \Theta_\sigma (T)
\endaligned
\end{equation}
for any $\sigma \in (0,1)$.
Using (\ref{6.3-7}) and integration by parts, we may bound $I_1$ in (\ref{6.3-5}) by
\begin{equation}\label{6.3-9}
\aligned
\left| \int_\Omega 
\text{\rm div} \big \{ \Delta f  (x/\varep) \nabla u_0\big\}  \cdot w_\varep\, dx \right|
 &+T^{-2} \int_\Omega |f(x/\varep)|\, |\nabla u_0| \, |\nabla w_\varep|\, dx\\
 &+ C \langle |\psi-\nabla \chi_T|\rangle  \| w_\varep\|_{L^2(\Omega)},
 \endaligned
\end{equation}
where we have used the fact
$
|\langle B_T\rangle |\le C \langle |\psi-\nabla \chi_T|\rangle.
$
Note that by (\ref{6.3-8}), the second term in (\ref{6.3-9}) is bounded by 
$C \, \Theta_1 (T)\, \| \nabla w_\varep\|_{L^2(\Omega)}$.

It remains to estimate the first term in (\ref{6.3-9}), which we denote by $I_{11}$.
To this end we write
$$
\aligned
\text{\rm div} \left\{ \Delta f (x/\varep) \nabla u_0\right\} \cdot w_\varep
&=\frac{\partial}{\partial x_i } \left\{ \Delta f_{ij}^{\alpha\beta} (x/\varep)\frac{\partial u_0^\beta}{\partial x_j}\right\}\cdot w_\varep^\alpha\\
&=\frac{\partial}{\partial x_i}
\left\{ \frac{\partial}{\partial x_k}
\left\{ \frac{\partial f_{ij}^{\alpha\beta}}{\partial x_k} -\frac{\partial f_{kj}^{\alpha\beta}}{\partial x_i} \right\} (x/\varep)
\frac{\partial u_0^\beta}{\partial x_j} \right\} \cdot w_\varep^\alpha\\
&\qquad +\frac{\partial}{\partial x_i}
\left\{ \frac{\partial^2 f_{kj}^{\alpha\beta}}{\partial x_k\partial x_i} (x/\varep) \frac{\partial u^\beta_0}{\partial x_j}\right\}
\cdot w_\varep^\alpha\\
&=-\frac{\partial}{\partial x_i}
\left\{ \varep
\left\{ \frac{\partial f_{ij}^{\alpha\beta}}{\partial x_k} -\frac{\partial f_{kj}^{\alpha\beta}}{\partial x_i} \right\} (x/\varep)
\frac{\partial^2 u_0^\beta}{\partial x_k\partial x_j} \right\} \cdot w_\varep^\alpha\\
&\qquad +\frac{\partial}{\partial x_i}
\left\{ \frac{\partial^2 f_{kj}^{\alpha\beta}}{\partial x_k\partial x_i} (x/\varep) \frac{\partial u^\beta_0}{\partial x_j}\right\}
\cdot w_\varep^\alpha,
\endaligned
$$
where we have used the product rule and the fact that
$$
\frac{\partial^2}{\partial x_i\partial x_k} \left\{
\left[ \frac{\partial f_{ij}^{\alpha\beta}}{\partial x_k} -\frac{\partial f_{kj}^{\alpha\beta}}{\partial x_i} \right] 
\big({x}/{\varep}\big)
\frac{\partial u_0^\beta}{\partial x_j} \right\} =0.
$$
It then follows from an integration by parts that
\begin{equation}\label{6.3.10}
\aligned
I_{11} & \le C \varep \int_\Omega |\nabla f (x/\varep)|\, |\nabla^2 u_0|\, |\nabla w_\varep|\, dx\\
&\qquad \qquad 
+C\sum_{j, \alpha, \beta}  \int_\Omega \big|\nabla \frac{\partial f^{\alpha\beta}_{kj}}{\partial x_k} (x/\varep)\big|\, |\nabla u_0|\, |\nabla w_\varep|\, dx\\
&= I_{11}^{(1)} +I_{11}^{(2)}.
\endaligned
\end{equation}
In view of (\ref{6.3-8}) we have
\begin{equation}\label{6.3.11}
I_{11}^{(1)}
\le C \, \varep \|\nabla f\|_{L^\infty} \|\nabla w_\varep \|_{L^2 (\Omega)}
\le C\, \Theta_\sigma (T)\, \|\nabla w_\varep\|_{L^2(\Omega)}
\end{equation}
for any $\sigma \in (0,1)$.

Finally, to estimate $I_{11}^{(2)}$, we note that
by the definition of $\chi_T$,
$$
\frac{\partial h_{ij}^{\alpha\beta}}{\partial y_i}
=\frac{\partial}{\partial y_i} \left\{ b_{T, ij}^{\alpha\beta}\right\}
=-\frac{1}{T^2} \chi_{T, j}^{\alpha\beta}.
$$
It follows that
$$
-\Delta\left( \frac{\partial f_{ij}^{\alpha\beta}}{\partial y_i}\right)
+\frac{1}{T^2} \left(\frac{\partial f_{ij}^{\alpha\beta}}{\partial y_i}\right)
=-\frac{1}{T^2} \chi_{T, j}^{\alpha\beta}.
$$
Observe that the function $T^{-1}\chi_T$ satisfies the assumption on $h$ in Lemma \ref{lemma-6.2}
with $\sigma =1$. As a result, we obtain 
$$
\Big\|  \nabla \frac{\partial f_{ij}^{\alpha\beta}}{\partial x_i}\Big\|_{L^\infty}
\le C_\sigma\,  \Theta_\sigma (T)
$$
for any $\sigma \in (0,1)$. This allows us to bound $I_{11}^{(2)}$
by $C_\sigma\, \Theta_\sigma (T) \| \nabla w_\varep\|_{L^2(\Omega)}$, and completes the proof.
\end{proof}

The next lemma gives an estimate for the norm of $v_\varep$ in $H^1(\Omega)$.

\begin{lemma}\label{lemma-6.4}
Let $v_\varep$ be the weak solution of (\ref{6.3-1-1}) with $T=\varep^{-1}$. Then
\begin{equation}\label{6.4-0}
\| v_\varep\|_{H^1(\Omega)} \le C_\sigma \left(T^{-1} \| \chi_T \|_{L^\infty}\right)^{\frac12-\sigma}
 \Big \{ \|\nabla u_0\|_{L^\infty(\Omega)}
 +\|\nabla^2 u_0\|_{L^2(\Omega)} \Big\}
\end{equation}
for any $\sigma\in (0,1/2)$, where $C_\sigma$ depends only on $A$, $\Omega$, and $\sigma$.
\end{lemma}

\begin{proof}
We may assume that $\|\nabla u_0\|_{L^\infty(\Omega)}
 +\|\nabla^2 u_0\|_{L^2(\Omega)}=1$. We may also assume that
 $\delta=T^{-1} \| \chi_T\|_{L^\infty}>0$, and is small, since $\delta\to 0$ as $T\to \infty$.
Choose  a cut-off function $\eta_\delta \in C_0^\infty (\rd)$ so that
$0\le \eta_\delta\le 1$, $\eta_\delta (x)=1$ if dist$(x, \partial\Omega)<\delta$,
$\eta_\delta (x)=0$ if dist$(x, \partial\Omega)\ge 2\delta$, and
$|\nabla\eta_\delta|\le C\, \delta^{-1}$.
Note that
\begin{equation}\label{6.3-14}
\aligned
\| v_\varep\|_{H^1(\Omega)}
& \le C\,\varep\, \| \chi_T (x/\varep)\nabla u_0\|_{H^{1/2}(\partial\Omega)}\\
 & \le C\,\varep\,  \| \eta_\delta \chi_T (x/\varep)\nabla u_0\|_{H^{1}(\Omega)}\\
 &\le C \left\{ \|\chi_T\|_{L^\infty} \delta^{-1/2}\varep
 + \left(\int_{\Omega_\delta} |\nabla \chi_T (x/\varep)|^2 \, dx\right)^{1/2}
 \right\},
 \endaligned
\end{equation}
where
$\Omega_\delta =\big\{ x\in \Omega:\, \text{\rm dist}(x,\partial\Omega)\le 2\delta\big\}$.
Since $\| \chi_T\|_{L^\infty} \delta^{-1/2}\varep
=\delta^{1/2}$, we only need to
estimate the integral of $|\nabla \chi_T (x/\varep)|^2$ over $\Omega_\delta$.

To this end, we 
 cover $\Omega_\delta$ with cubes  $\{ Q_j\}$ of side length $\delta$ so that
 $\sum_j |Q_j|\le C \, \delta$.
 It follows that
 \begin{equation}\label{6.3-15}
 \aligned
 \int_{\Omega_\delta} |\nabla \chi_T (x/\varep)|^2\, dx
 &\le \sum_j \int_{Q_j}  |\nabla \chi_T (x/\varep)|^2\, dx
 \le \sum_j |Q_j| \average_{\frac{1}{\varep} Q_j}
 |\nabla \chi_T|^2\\
 &\le C\, \delta \sup_{\ell (Q)=\delta T} \average_Q |\nabla \chi_T|^2
 \le C_\sigma\, \delta^{1-\sigma}
\endaligned
\end{equation}
for any $\sigma\in (0,1)$,
where we have used the estimate (\ref{5.2-1}) in the last inequality.
This, together with (\ref{6.3-14}), gives (\ref{6.4-0}).
\end{proof}

We are now in a position to give the proof of Theorems \ref{main-theorem-1} and \ref{main-theorem-2}.

\begin{proof}[\bf Proof of Theorem \ref{main-theorem-1}]

It follows from Theorem \ref{theorem-6.1} and Lemma \ref{lemma-6.4} that for 
any $\sigma \in (0,1)$ and $\delta\in (0,1/2)$,
\begin{equation}\label{6.6-1}
\aligned
\| u_\varep -u_0 - & \varep\chi_T (x/\varep)\nabla u_0\|_{H^1(\Omega)}\\
 &\le C\, \Big\{ \Theta_\sigma (T) +\langle |\psi-\nabla \chi_T|\rangle \Big\} \| u_0\|_{W^{2,2}(\Omega)}\\
& \quad \qquad
+C \big[\Theta_\sigma (T)\big]^{\frac12 -\delta}
\Big\{ \|\nabla u_0\|_{L^\infty(\Omega)} +\|\nabla^2 u_0\|_{L^2(\Omega)} \Big\}\\
&\le C\,  \left\{ 
\langle |\psi-\nabla \chi_T|\rangle 
+\big[\Theta_\sigma (T)\big]^{\frac12 -\delta}
\right\}
\| u_0\|_{W^{2, p} (\Omega)}\\
&\le  C\, \left\{ 
\langle |\psi-\nabla \chi_T|\rangle 
+\big[\Theta_1 (T)\big]^{\sigma (\frac12 -\delta)}
\right\}
\| u_0\|_{W^{2, p} (\Omega)},
\endaligned
\end{equation}
where $T=\varep^{-1}$ and we have used the Sobolev imbedding $\|\nabla u_0\|_{L^\infty(\Omega)}
\le C\, \| u_0\|_{W^{2, p}(\Omega)}$ for $p>d$.
This implies that
$$
\aligned
\| u_\varep -u_0\|_{L^2(\Omega)}
&\le \| \varep\chi_T (x/\varep)\nabla u_0\|_{L^2(\Omega)}
+ C \left\{ 
\langle |\psi-\nabla \chi_T|\rangle 
+\big[\Theta_1 (T)\big]^{\frac14}
\right\}
\| u_0\|_{W^{2, p} (\Omega)}\\
&\le C\, 
 \left\{ 
\langle |\psi-\nabla \chi_T|\rangle 
+\big[\Theta_1 (T)\big]^{\frac14}
\right\}
\| u_0\|_{W^{2, p} (\Omega)},
\endaligned
$$
where $C$ depends only on  $A$ and $\Omega$.
Since $\langle |\psi-\nabla\chi_T|\rangle  +\big[\Theta_1 (T)\big]^{\frac14}
\to 0$ as $T\to \infty$,
one may find a modulus $\eta$ on $(0,1]$, depending only on $A$, such that $\eta (0+)=0$ and
$$
\langle |\psi-\nabla \chi_T|\rangle 
+\big[\Theta_1 (T)\big]^{\frac14}
\le \eta (T^{-1})
$$
for $T\ge 1$.
As a result, we obtain 
$$
\aligned
\| u_\varep -u_0 - \varep\chi_T (x/\varep)\nabla u_0\|_{H^1(\Omega)}
&\le C\, \eta (\varep) \, \| u_0\|_{W^{2, p} (\Omega)},\\
\| u_\varep -u_0 \|_{L^2(\Omega)}
&\le C\, \eta (\varep) \,\| u_0\|_{W^{2, p} (\Omega)}.
\endaligned
$$
Finally, we observe that by Theorem \ref{main-theorem-3}, for any $\sigma \in (0,1)$,
$$
\aligned
\| u_\varep\|_{C^\sigma (\overline{\Omega})}
&\le C \, \Big\{ \| g\|_{C^\sigma (\partial\Omega)} +\| F\|_{L^d (\Omega)} \Big\}\\
&\le C\, \left\{ \| u_0\|_{C^\sigma(\overline{\Omega})} +\|\nabla^2 u_0\|_{L^d(\Omega)} \right\}\\
&\le C\, \| u_0\|_{W^{2, d} (\Omega)}.
\endaligned
$$
It follows by interpolation that for any $\sigma \in (0,1)$,
$$
\| u_\varep -u_0 \|_{C^\sigma (\overline{\Omega})} \le C\, \widetilde{\eta} (\varep)\, \| u_0\|_{W^{2, p}(\Omega)},
$$
where $\widetilde{\eta}$ is a modulus function depending only on $A$ and $\sigma$, and $\widetilde{\eta} (0+)=0$.
This complete the proof.
\end{proof}

\begin{proof}[\bf Proof of Theorem \ref{main-theorem-2}]
Estimate (\ref{1.2-1}) follows directly from  (\ref{6.6-1}) and Theorem \ref{theorem-5.6}.
To see (\ref{1.2-0}) we use
\begin{equation}\label{6.7-1}
\aligned
\| u_\varep -u_0\|_{L^2(\Omega)}
&\le \| u_\varep -u_0 -\varep \chi_T (x/\varep)\nabla u_0 + v_\varep\|_{L^2(\Omega)}
+\| v_\varep\|_{L^2(\Omega)}\\
&\le C_\sigma \Big\{ \Theta_\sigma (T) +\langle |\psi -\nabla \psi_T |\rangle  \Big\} \| u_0\|_{W^{2,2}(\Omega)}
+\| v_\varep\|_{L^2(\Omega)},
\endaligned
\end{equation}
where $v_\varep$ is defined in Theorem \ref{theorem-6.1}.
By Theorem \ref{main-theorem-3} we obtain 
$$
\aligned
\| v_\varep\|_{L^2(\Omega)}
&\le C\, \| v_\varep \|_{L^\infty(\Omega)}\\
&\le C\, \| \varep \chi_T (x/\varep) \nabla u_0\|_{C^{\sigma_1}(\partial\Omega)}\\
&\le C\Big\{ \varep^{1-\sigma_1} \|\chi_T \|_{C^{0,\sigma_1}}  +\Theta_\sigma (T)\Big\} \| \nabla u_0\|_{C^{\sigma_1}
(\partial\Omega)}\\
&\le C\Big\{ T^{\sigma_1 -1} \| \chi_T \|_{C^{0, \sigma_1}}
  +\Theta_\sigma (T)\Big\} \|  u_0\|_{W^{2, p}(\Omega)},
\endaligned
$$
where $p>d$, $\sigma\in (0,1)$ and $0<\sigma_1<1-\frac{d}{p}$.
Since $ T^{-1} \| \chi_T \|_{L^\infty} \le C_\sigma\, \Theta_\sigma (T)$
and $|\chi_T (x)-\chi_T (y)|\le C_\alpha\, T^{1-\alpha} |x-y|^\alpha$ for any $\alpha\in (0,1)$,
it follows by interpolation that
$$
T^{\sigma_1 -1} \|\chi_T \|_{C^{0, \sigma_1}}
\le C \left[ \Theta_\sigma (T)\right]^{1-\sigma_2}
$$
for any $\sigma_2>\sigma_1$.
Hence,
$$
\| v_\varep\|_{L^2(\Omega)}
\le C\, \big[ \Theta_\sigma (T)\big]^{1-\delta} \| u_0\|_{W^{2, p}(\Omega)}
\le C\, \big[ \Theta_1 (T)\big]^{\sigma (1-\delta)} \| u_0\|_{W^{2, p}(\Omega)}
$$
for any $\delta, \sigma \in (0,1)$ and $p>d$,
where $C$ depends only on $\delta$, $p$, $\sigma$, $A$ and $\Omega$.
This, together with (\ref{6.7-1}) and Theorem \ref{theorem-5.6}, gives
$$
\aligned
\| u_\varep -u_0\|_{L^2(\Omega)}
&\le C\, \Big\{ \langle |\psi -\nabla \chi_T|\rangle 
+ \big[\Theta_1 (T)\big]^\sigma \Big\} \| u_0\|_{W^{2, p} (\Omega)}\\
& \le C\left\{\int_{\frac{1}{2\varep}}^\infty \frac{\Theta_\sigma (r)}{r}\, dr
+\big[ \Theta_1 (\varep^{-1})\big]^{\sigma} \right\}
\| u_0\|_{W^{2, p}(\Omega)}
\endaligned
$$
for any $\sigma\in (0,1)$, and completes the proof.
\end{proof}



\section{Quasi-periodic coefficients}
\setcounter{equation}{0}

In this section we consider the case where $A(x)$ is quasi-periodic and continuous.
More precisely, without loss of generality, we will assume that
\begin{equation}\label{quasi}
\left\{
\aligned
& A(x)=B(j_\lambda (x)),\\
& B \text{ is 1-periodic  and continuous in } \mathbb{R}^M,
\endaligned
\right.
\end{equation}
where $M=m_1+m_2+\cdots +m_d$, and for $x=(x_1, x_2, \dots, x_d)\in \rd$,
$$
j_\lambda (x)=
(\lambda^1_1 x_1,\lambda_1^2 x_1, \dots, \lambda_1^{m_1} x_1, 
\lambda_2^{1} x_2, \dots, \lambda_2^{m_2} x_2, 
\dots, \lambda_d^1 x_d, \dots, \lambda^{m_d}_d x_d)\in \mathbb{R}^M.
$$
Also, for each $i=1,2\dots, d$, the set $\{ \lambda_i^1, \dots, \lambda_i^{m_i}\} $ is assumed to be
linearly independent over $\mathbb{Z}$.
Under these conditions it is known that $A(x)$ is uniformly almost-periodic.
We shall be interested in conditions on $\lambda=\big(\lambda_i^j\big)$ that implies
the power decay of $\rho (R)$ as $R\to \infty$.
For convenience we consider
\begin{equation}\label{rho-1}
\rho_1 (R)
=\sup_{y\in \rd}\inf_{\substack{ z\in \rd \\ \| z\|_\infty\le R }}
\| A(\cdot +y ) -A(\cdot +z)\|_{L^\infty},
\end{equation}
where $\| z\|_\infty=\max (|z_1|, \dots, |z_d|)$ for $z=(z_1, \dots, z_d)$.
It is easy to see that $\rho_1(\sqrt{d} R) \le \rho (R) \le \rho_1 (R)$.

Let 
$$
\omega (\delta)=\sup \big\{ |B(x)-B(y)|: \, \| x-y\|_\infty\le \delta \big\},\quad \delta>0
$$
denote the modulus of continuity of $B(x)$.
For $x\in \mathbb{R}$, write $x=[x] +<x>$, where $[x]\in \mathbb{Z}$ and
$<x> \in [-1/2, 1/2)$. If $x=(x_1, \dots, x_M) \in \mathbb{R}^M$, we define $[x]
=([x_1], \dots, [x_N])$ and $<x>=(<x_1>, \dots, <x_M>)$.
It is easy to see that $\|<x>\|_\infty$ gives the distance from $x$ to $\mathbb{Z}^M$ with respect to the
norm $\|\cdot\|_\infty$.

\begin{lemma}\label{lemma-7.1}
Let $\rho_1 (R)$ be defined by (\ref{rho-1}). Then, for any $R>0$,
$\rho_1 (R)\le \omega (\theta_\lambda (R))$, where
\begin{equation}\label{7.1-0}
\theta_\lambda (R)
=\sup_{x\in [-1/2, 1/2]^M} \inf_{\substack{z\in \mathbb{R}^d  \\ \|z\|_\infty\le R}}
\| x-< j_\lambda( z )>\|_\infty.
\end{equation}
\end{lemma}

\begin{proof}
Note that, since $B$ is 1-periodic,
$$
\aligned
|B(x)-B(y)| &= |B(y +[x-y] +<x-y>)- B(y)|\\
&= | B(y+<x-y>) -B(y)|\\
&\le \omega (\| <x-y>\|_\infty)
\endaligned
$$
for any $x,y\in \mathbb{R}^M$.
It follows that
$$
|A(x+y)-A(x+z)|\le \omega (\|<j_\lambda (y)-j_\lambda (z)>\|_\infty),
$$
for any $x,y,z\in \rd$.
This implies that
$$
\rho_1 (R)
\le \sup_{y\in \rd} \inf_{\substack{z\in \rd\\ \|z\|_\infty\le R}}
\omega ( \|<j_\lambda (y)-j_\lambda (z) >\|_\infty).
$$
Using 
 $$
 \aligned
 \|<j_\lambda (y) -j_\lambda (z)>\|_\infty
 &=\| < <j_\lambda (y)>-<j_\lambda (z)>> \|_\infty\\
& \le \|<j_\lambda (y)> -<j_\lambda (z)> \|_\infty,
\endaligned
$$
we obtain 
$$
\rho_1 (R)
\le \sup_{y\in \rd} \inf_{\substack{z\in \rd\\ \|z\|_\infty\le R}}
\omega ( \| <j_\lambda (y)> -<j_\lambda (z)>
\|_\infty)
\le \omega \big(\theta_\lambda (R) \big),
$$
where we have used the continuity of $\omega(\delta)$ for the second inequality.
\end{proof}

Let $\lambda_i =(\lambda_i^1, \lambda_i^2, \dots, \lambda_i^{m_i})\in \mathbb{R}^{m_i}$
for each $1\le i\le d$, and 
$$
j_{\lambda_i} (t)= (\lambda_i^1 t, \lambda_i^2 t, \dots, \lambda_i^{m_i} t)\in \mathbb{R}^{m_i}
\quad \text{ for } t\in \mathbb{R}.
$$
Thus, for $z=(z_1, z_2, \dots, z_d)\in \rd$,
$$
j_\lambda (z)=( j_{\lambda_1} (z_1), j_{\lambda_2} (z_2), \dots, j_{\lambda_d} (z_d)).
$$
It follows that
$$
\| x-<j_\lambda (z)>\|_\infty
=\max_{1\le i\le d} \| x_i - <j_{\lambda_i} (z_i)>\|_\infty,
$$
where $x=(x_1, x_2, \dots, x_d)\in \mathbb{R}^M$ and $x_i\in \mathbb{R}^{m_i}$.
This implies that
\begin{equation}\label{7.1-10}
\theta_\lambda (R)=\max_{1\le i\le d} \theta_{\lambda_i} (R),
\end{equation}
where
\begin{equation}\label{7.1-11}
\theta_{\lambda_i} (R)
=\sup_{x\in [-1/2, 1/2]^{m_i}}
\inf_{\substack{t\in \mathbb{R}\\ |t|\le R}}
\| x- <j_{\lambda_i} (t)>\|_\infty.
\end{equation}

Note that if $m_i=1$, then $\theta_{\lambda_i} (R)=0$ for $R$ large.
We will use the Erd\"os-Turan-Koksma inequality in the discrepancy theory to estimate the function 
$\theta_{\lambda_i} (R)$, defined by (\ref{7.1-11}), for $m_i\ge 2$.

Let $P=P_N=\{ x_1, x_2, \dots, x_N\}$ be a finite subset of $[-1/2, 1/2]^m$.
The discrepancy of $P$ is defined as
$$
D_N (P)=\sup_{B} \left| \frac{A(B; P)}{N} -|B|\right|,
$$
where the supremum is taken over all rectangular boxes
$B=[a_1, b_1]\times \cdots [a_m, b_m] \subset [-1/2, 1/2]^m$,
and $A(B; P)$ denotes the number of elements of $P$ in $B$.
It follows from the Erd\"os-Turan-Koksma inequality that
\begin{equation}\label{Erdos}
D_N (P)
\le C \left\{ \frac{1}{H}
+\sum_{\substack {n\in \mathbb{Z}^m\\ 0< \| n\|_\infty\le H}}
\frac{1}{(1+|n_1|) \cdots (1+|n_m|)}
\left| \frac{1}{N}\sum_{x\in P}
e^{2\pi i (n\cdot  x)} \right| \right\}
\end{equation}
for any $H\ge 1$, where $C$ depends only on $m$ (see e.g. \cite[p.15]{Tichy}).
It is not hard to see that
\begin{equation}\label{7.13}
\max_{y\in [-1/2, 1/2]^m}\min_{z\in P_N}
\| y-z\|_\infty
\le \frac12 \left[ D_N (P_N)\right]^{\frac{1}{m}}.
\end{equation}

\begin{lemma}\label{lemma-7.2}
Let $R\ge 2$ and $\ell\ge 2$ be two positive integers.
We divide the interval $[-R, R]$ into $2R\ell$ subintervals of length $1/\ell$.
Let $N=2R\ell$ and
$$
P_N=\left\{ x=<j_\lambda (t)> \in [-1/2, 1/2]^m: \  t=j +\frac{k}{\ell}, \ -R\le j\le R-1
\text{ and } 0\le k\le \ell-1 \right\},
$$
where $\lambda=(\lambda^1, \dots, \lambda^m)\in \mathbb{R}^m$ and $m\ge 2$.
Suppose that  there exist $c_0>0$ and  $\tau>0$ such that
\begin{equation}\label{7.2-0}
| n\cdot \lambda|
\ge c_0\, |n|^{-\tau} \quad \text{ for any } n\in \mathbb{Z}^m\setminus \{ 0\}.
\end{equation}
Then
\begin{equation}\label{7.2-0-0}
D_N (P_N)\le C\, \left\{ R^{-\frac{1}{\tau +1}} (\log R)^{m-1} +N^{-1} R^{1+\frac{1}{\tau +1}} \left(\log R\right)^{m-1}\right\},
\end{equation}
where $C$ depends only on $m$, $c_0$, $|\lambda|$ and $\tau$.
\end{lemma}

\begin{proof}
 Let $f(t) =e^{2\pi i (n\cdot \lambda) t}$ and
\begin{equation}\label{7.2-1}
I_n =\frac{1}{N}\sum_{x\in P_N} e^{2\pi  i (n\cdot x)}
=\frac{1}{N}\sum_{j,k} f (t_{jk}),
\end{equation}
where $n\in \mathbb{Z}^m\setminus \{ 0\}$,
$j=-R, \dots R-1$, $k=0, \dots, \ell-1$ and $t_{jk}=j +\frac{k}{\ell}$.
Using
$$
\left| \frac{1}{2R}\int_{-R}^R f (t)\, dt
-\frac{1}{N}\sum_{j,k} f (t_{jk}) \right|
\le C\, \ell^{-1}  \| f^\prime\|_\infty,
$$
we obtain
$$
\aligned
|I_n|
&\le  C\ell^{-1} \| f^\prime\|_\infty + \left|\frac{1}{2R} \int_{-R}^R f(t)\, dt \right|\\
&\le C\, \ell^{-1} | n \cdot \lambda| +\frac{C}{R | n\cdot \lambda|}\\
&\le C \left\{ \ell^{-1} | n| + R^{-1} |n|^\tau \right\},
\endaligned
$$
where we have used the assumption (\ref{7.2-0}).
In view of (\ref{Erdos}), we obtain
$$
\aligned
D_N (P_N)
&\le C \left\{ \frac{1}{H}
+ \sum_{\substack { n\in \mathbb{Z}^m\\ 0<\| z\|_\infty\le H}}
\frac{\ell^{-1} |n| +R^{-1} |n|^\tau}{(1+|n_1|)\cdots (1+|n_m|)}\right\}\\
&\le C \left\{ \frac{1}{H}
+ \int_{|x| \le C\, H} 
\frac{ \ell^{-1} |x| +R^{-1} |x|^\tau }{(1+|x_1|) \cdots (1+|x_m|) }\, dx\right\}\\
&\le C\left\{ \frac{1}{H}
+R N^{-1} H \left( \log H\right)^{m-1} + R^{-1} H^\tau \left( \log H \right)^{m-1} \right\}
\endaligned
$$
for any $H\ge 2$.
By taking $H=R^{\frac{1}{\tau +1}}$, we obtain the estimate (\ref{7.2-0-0}).
\end{proof}

\begin{theorem}\label{theorem-7.3}
Let $\lambda=(\lambda_1, \dots, \lambda_d)$ with $\lambda_i =(\lambda_i^1, \dots, \lambda_i^{m_i})
\in \mathbb{R}^{m_i}$ for $1\le i\le d$.
Suppose that there exist $c_0>0$ and $\tau>0$ such that for each $1\le i\le d$ with $m_i\ge 2$, 
\begin{equation}\label{7.3-0}
|n \cdot \lambda_i| \ge c_0\, |n|^{-\tau} 
\quad \text{ for any } n\in \mathbb{Z}^{m_i}\setminus \{ 0\}.
\end{equation}
Then, for any $R\ge 2$,
\begin{equation}\label{7.3-0-0}
\theta_\lambda (R) \le C\, R^{-\frac{1}{\widetilde{m} (\tau +1)}}
\left(\log R\right)^{1-\frac{1}{\widetilde{m}}},
\end{equation}
where $\widetilde{m}=\max\{ m_1, \dots, m_d\} $ and
$C$ depends only on $d$, $\widetilde{m}$, $c_0$ and $\tau$.
\end{theorem}

\begin{proof}
Suppose $m_i\ge 2$. Let $P=P_N$ be same as in Lemma \ref{lemma-7.2}.
It follows from  (\ref{7.13})  and Lemma \ref{lemma-7.2}  that
$$
\aligned
\theta_{\lambda_i} (R)  
&\le C \left\{ R^{-\frac{1}{\tau +1}} \left( \log R\right)^{m_i-1} 
+N^{-1} R ^{1+\frac{1}{\tau +1}} \left( \log R\right)^{m_i-1} \right\}^{\frac{1}{m_i}}\\
&\le C\, R^{-\frac{1}{m_i (\tau +1)}} \left( \log R\right)^{1-\frac{1}{m_i}},
\endaligned
$$
where we have taken $N=C\, R^{1+\frac{2}{\tau +1}}$.
This, together with (\ref{7.1-10}), gives (\ref{7.3-0-0}).
\end{proof}

\begin{remark}\label{remark-7.1}
{\rm
Suppose that $A(x)=B(j_\lambda (x))$ and $B(y)$ is 1-periodic.
Also assume that $\lambda$ satisfies the condition (\ref{7.3-0})
and $B(y)$ is H\"older continuous of order $\alpha$ for some $\alpha \in (0,1]$.
It follows from Lemma \ref{lemma-7.1} and Theorem \ref{theorem-7.3} that
\begin{equation}
\rho (R) \le C\, R^{-\frac{\alpha}{\widetilde{m} (\tau +1)}}
\left(\log R\right)^{\alpha (1-\frac{1}{\widetilde{m}})}
\end{equation}
for $R\ge 1$.
In view of Remark \ref{remark-1.1} this leads to 
$$
\| u_\varep -u_0\|_{L^2(\Omega)}
\le C_\gamma\,  \varep^\gamma\,  \| u_0\|_{W^{2, p}(\Omega)}
$$
for any $0<\gamma < \frac{\alpha}{\alpha +\widetilde{m} (\tau +1)}$.
We point out that if $A(y)$ satisfies  the condition (\ref{7.3-0}) and is sufficiently smooth, 
the sharp estimate $\| u_\varep -u_0\|_{L^2(\Omega)} =O(\varep)$ was obtained in \cite{Kozlov-1979}.
}
\end{remark}

\bibliography{s33.bbl}

\medskip

\begin{flushleft}
Zhongwei Shen, 

 Department of Mathematics
 
University of Kentucky

Lexington, Kentucky 40506,
USA. 

 Fax: 1-859-257-4078.

E-mail: zshen2@uky.edu
\end{flushleft}

\medskip

\noindent \today

\end{document}